\documentclass[a4paper,12pt]{article}

\usepackage[top=1in,left=1in,right=1in,bottom=1.5in]{geometry}

\usepackage[T1]{fontenc}
\usepackage{amsmath,amssymb,amsfonts,amsthm,mathrsfs,eucal,dsfont,verbatim}

\usepackage[english]{babel}

\usepackage{color}

\newcommand{\normal}{\color{black}}

\newcommand{\ax}{\mathbb{R}^+}

\newcommand{\rd}{{\mathbb{R}^d}}
\newcommand{\rone}{\mathbb{R}}
\renewcommand{\Re}{\rone}
\renewcommand{\star}{\circledast}

\newcommand{\E}{\mathds{E}}
\renewcommand{\P}{\mathds{P}}

\newcommand{\Beta}{\mathrm{B}}

\newcommand{\prt}{\partial}
\renewcommand {\epsilon}{\varepsilon}

\newcommand{\DD}{\mathbb{D}}

\newcommand{\eps}{\varepsilon}

\newcommand{\Ee}{\mathds E}

\theoremstyle{plain}
\newtheorem{thm}{Theorem}[section]
\newtheorem{prop}{Proposition}[section]

\newtheorem{lem}{Lemma}[section]
\theoremstyle{definition}
\newtheorem{dfn}{Definition}[section]

\theoremstyle{remark}
\newtheorem{rem}{Remark}[section]

\usepackage{hyperref}

\DeclareMathSymbol{\ophi}{\mathalpha}{letters}{"1E}

\renewcommand{\phi}{\varphi}

\newcommand{\be}{\begin{equation}}
\newcommand{\ee}{\end{equation}}
\newcommand{\ben}{\begin{equation*}}
\newcommand{\een}{\end{equation*}}

\newcommand{\ba}{\begin{aligned}}
\newcommand{\ea}{\end{aligned}}

\newfont{\cyrfnt}{wncyr10}
\def\J3{\cyrfnt{\rm \u{\cyrfnt I}}}
\def\j3{\cyrfnt{\rm \u{\cyrfnt i}}}

\usepackage[]{color}
\definecolor{DarkGreen}{rgb}{0.1,0.7,0.3}   

\definecolor{DarkGreen}{rgb}{0.1,0.7,0.3}   

\numberwithin{equation}{section}

\allowdisplaybreaks[4]

\begin{document}

\title{ Parametrix construction of the transition probability density of the solution to an SDE driven by $\alpha$-stable noise}

\author{%
    \textsc{Victoria Knopova}%
    \thanks{  V.M.\ Glushkov Institute of Cybernetics,
            NAS of Ukraine,
            40, Acad.\ Glushkov Ave.,
            03187, Kiev, Ukraine,
            \texttt{vic\underline{ }knopova@gmx.de}}
    \textrm{\ \ and\ \ }
    \stepcounter{footnote}\stepcounter{footnote}\stepcounter{footnote}
    \stepcounter{footnote}\stepcounter{footnote}%
    \textsc{Alexei Kulik}%
    \thanks{Institute of Mathematics, NAS of Ukraine, 3, Tereshchenkivska str., 01601  Kiev, Ukraine,
    \texttt{kulik@imath.kiev.ua}}
    }

\date{}

\maketitle

\begin{abstract}
    \noindent

    Let $L:= -a(x) (-\Delta)^{\alpha/2}+ (b(x), \nabla)$, where  $\alpha\in (0,2)$, and  $a:\rd\to (0,\infty)$,  $b: \rd\to \rd$. Under certain regularity assumptions on the coefficients $a$ and $b$, we  associate with  the $C_\infty(\rd)$-closure of $(L, C_\infty^2(\rd))$  a Feller Markov process  $X$, which possesses a  transition probability density $p_t(x,y)$.
     To construct this transition probability density and to obtain the two-sided  estimates on it, we develop a new version of the parametrix method,  which even allows  us to handle the case  $0<\alpha\leq 1$ and $b\neq 0$, i.e. when the gradient part of the generator is not dominated by the jump part.


\end{abstract}

\section{Introduction}\label{s12}

 Let $Z^{(\alpha)}$, $\alpha\in (0,2)$,  be a symmetric $\alpha$-stable process in $\Re^d$;  that is, a L\'evy process  with the characteristic function
$$
\E e^{i(\xi, Z_t^{(\alpha)})}=e^{-t|\xi|^\alpha}, \quad \xi\in \Re^d.
$$
 It is well known that the generator $L^{(\alpha)}$ of the semigroup
 $(P_t^{(\alpha)})_{t\geq 0}$, where
 $$
 P_t^{(\alpha)} f(x)= \Ee^x f(Z^{(\alpha)}_t),
  $$
 admits on  $C_\infty^2(\rd)$ the representation
\begin{equation}\label{pv}
L^{(\alpha)}f(x)= \hbox{P.V.} \int_{\Re^d}\Big(f(x+u)-f(x)\Big)\frac{c_\alpha}{|u|^{d+\alpha}}du.
\end{equation}
Here and below we denote by   $C_\infty^k(\rd)$, $k\geq 0$,   the space of $k$ times  continuously differentiable functions, vanishing at infinity together with their  derivatives.
The operator $L^{(\alpha)}$ is also called  a  \emph{fractional Laplacian}, and is denoted by $-(-\Delta)^{\alpha/2}$.

Consider the following ``perturbation'' of the operator $L^{(\alpha)}$:
\begin{equation}\label{symbol}
Lf(x)=a(x) L^{(\alpha)}f(x)+\Big(b(x), \nabla f(x)\Big), \quad f\in C_\infty^2(\Re^d),
\end{equation}
where $a(\cdot)>0$, $b(\cdot)\in \rd$.
When the coefficients $a(\cdot)$ and $b(\cdot)$ do not depend on $x$, the operator $L$ is just the  restriction to $C_\infty^2(\Re^d)$ of the  generator of the semigroup $\{T_t,\, t\geq 0\}$, which corresponds to the L\'evy process $Z^{(\alpha)}$ re-scaled by $a$ and with drift $b$. 

The general case is much more complicated, and the purpose of this paper is to show that under suitable assumptions on the coefficients and the parameter $\alpha$, the $C_\infty$-closure of the  operator $(L,C_\infty^2(\rd))$ is the generator of  a semigroup $\{P_t, t\geq 0\}$, which corresponds to a strong Markov process.

Our approach is  analytic, and mainly relies on  the \emph{parametrix construction} of  the ``candidate''  $p_t(x,y)$ for the  transition probability density of the required Markov process. We develop a new version of the parametrix method, which substantially depends on  the relation between the regularity of the drift coefficient $b(x)$ and the parameter $\alpha$ and, in particular,
allows us to handle the case  $0<\alpha\leq 1$ and $b\neq 0$, i.e. the one where the gradient part of the generator is not dominated by the jump part. To associate the constructed kernel $p_t(x,y)$ with a  Markov process in a unique way, we develop a new method, which we believe may be useful in other settings as well. This method relies on the fact that $L$ possesses the positive maximum principle, and exploits a new notion of an \emph{approximate fundamental solution}. We refer the reader to a detailed discussion  in Section~\ref{over}, where we also give an overview of available results.

We also consider the probabilistic counterpart to the problem described above. Namely, we consider an SDE  driven by $Z^{(\alpha)}$
\be\label{SDE}
dX_t=b(X_t)\, dt+\sigma(X_{t-})\, dZ^{(\alpha)}_t;
\ee
here and below we denote $\sigma(x)=a^{1/\alpha}(x)$. Using the parametrix construction and the fact that the closure of $(L,C_\infty^2(\rd))$ is the generator of a Markov process, we show that a  weak solution to \eqref{SDE} is unique and actually coincides with this Markov process. This fact also ensures that the martingale problem  for $(L,C_\infty^2(\rd))$ is well posed.
Finally,  we provide lower and upper bounds for the   transition probability density $p_t(x,y)$.

The paper is organized as follows. In Section~\ref{s20} we  formulate the main results, give the outline of the proofs and an overview of already existing results, comparing them with ours. Section~\ref{s2} is devoted to the parametrix method for construction of   the function $p_t(x,y)$, which is the candidate for being the fundamental solution to the Cauchy problem for $\prt_t -L$.  Section~\ref{s4} is devoted to the relation between the operators  $P_t$ and  $L$,  in particular, we prove that the family of operators $\{P_t,\, t\geq 0\}$ forms a strongly continuous contraction semigroup on $C_\infty(\rd)$. Then we  prove  that the extension $(A,D(A))$ of $(L, C_\infty^2(\rd))$ is in fact the generator of the semigroup $\{P_t$, $t\geq 0\}$, and, moreover, that $p_t(\cdot ,y)\in D(A)$, and is the fundamental solution to the Cauchy problem for $\prt_t -A$.   In Section~\ref{sweak} we prove that the constructed process $X$ is the weak solution to \eqref{SDE},  and that the martingale problem $(L, C_\infty^2(\rd))$  is well-posed.   In Section~\ref{s3} we give the estimates on the time derivative $\prt_t p_t(x,y)$ and related auxiliary function  appearing in the parametrix construction.   In Section~\ref{s7} we give the proofs of lower and upper bounds on $p_t(x,y)$. Appendices  A and B  contain some auxiliary results, used in the proofs.

\section{The main results: preliminaries, formulation, and discussion}\label{s20}

\subsection{Notation and preliminaries}\label{sA}

 Through the paper we use the following notation.

 By   $g^{(\alpha)}(x)$  we denote  the distribution density of the symmetric $\alpha$-stable variable $Z_1^{(\alpha)}$.
 Note that $L^{(\alpha)}$ is a homogeneous operator of the order $\alpha$ and the process $Z^{(\alpha)}$ is self-similar: for any $c>0$, the process
 $$
 c^{-1/\alpha}Z_{ct}^{(\alpha)},\quad t\geq  0,
 $$
 has the same law as $Z^{(\alpha)}$. Consequently, the transition probability density of $Z^{(\alpha)}$ equals $t^{-d/\alpha} g^{(\alpha)}(t^{-1/\alpha} (y-x))$.
 By $C_\infty(\rd)$ (respectively, $C_b(\rd)$)  we denote the class of continuous functions vanishing at infinity (respectively, bounded); clearly, $C_\infty(\rd)$ is a Banach space with respect to the  $\sup$-norm $\|\cdot\|_\infty$.  By $C_\infty^k(\rd)$ (respectively, $C_b^k(\rd)$), $k\geq 1$,  we denote  the  class of $k$-times continuously differentiable functions vanishing  at infinity  (respectively, bounded) together with their derivatives.

We use the following notation for space and time-space convolutions of functions:
$$
(f\ast g)_t(x,y):=\int_{\Re^d}f_{t}(x,z)g_{t}(z,y)\, dz,
$$
$$
(f\star g)_t(x,y):=\int_0^t\int_{\Re^d}f_{t-s}(x,z)g_{s}(z,y)\, dzds.
$$

As usual,    $a\wedge b :=\min(a,b)$, $a \vee b:=\max(a,b)$. By $|\cdot|$ we denote both the modulus of a real number and the Euclidean norm of a vector.  By  $c$  and $C$ we denote positive constants, the value of which may vary from place to place.   Relation $f\asymp g$ means that
$$
cg\leq f\leq C g.
$$
By $\Gamma(\cdot), \Beta(\cdot, \cdot)$ we denote the Euler Gamma- and Beta-functions. Finally, we write  $L_x $ to emphasize that the operator $L$ acts on a function $f(x,y)$ with respect to the variable $x$,   i.e.,  $L_xf(x,y)= Lf(\cdot,y)(x)$.

Recall that a real-valued  function $p_t(x,y)$ is said to be the \emph{fundamental solution} to the Cauchy problem for the operator
\be\label{L_ful}
\prt_t-L,
\ee
if for  $t>0$ it is differentiable in $t$, belongs to the domain of $L$ as a function of $x$, and satisfies
\be\label{L_fund}
\Big(\prt_t-L_x\Big)p_t(x,y)=0, \quad t>0, \quad x,y\in \Re^d,
\ee
\be\label{L_delta}p_t(x, \cdot) \Rightarrow \delta_x, \quad t\to 0+, \quad x\in \Re^d;
\ee
see  \cite[Def.~2.7.12]{Ja02} in the case of a general pseudo-differential operator, which is the generalization of the corresponding definition (cf. \cite{Fr64}, for example) in the parabolic/elliptic setting.

In order to simplify the further exposition,  we briefly outline the  \emph{parametrix method}, which we use to construct $p_t(x,y)$.  Consider \emph{some} approximation $p_t^0(x,y)$ to this function, and denote by $r_t(x,y)$ the residue term  with respect to  this approximation:
\be\label{sol}
p_t(x,y)=p_t^0(x,y)+r_t(x,y).
\ee
Put
\be\label{Phi}
\Phi_t(x,y):=-\Big(\prt_t-L_x\Big)p_t^0(x,y),\quad t>0, \quad x,y\in\Re^d.
\ee
Recall that $p_t(x,y)$ is supposed   to be the fundamental solution for the operator  (\ref{L_ful}), hence
\begin{equation}\label{L_fund2}
\Big(\prt_t-L_x\Big)r_t(x,y)=\Phi_t(x,y).
\end{equation}
 Recall that if $p_t(x,y)$ is the fundamental solution to  \eqref{L_fund}, then one expects the solution  to  equation  \eqref{L_fund2} to be of the form
$$
r_t(x,y)=(p\star \Phi)_t(x,y).
$$
Substituting now in the right-hand side of the above equation  representation \eqref{sol} for $p_t(x,y)$, we get the following equation for $r_t(x,y)$:
$$
r_t(x,y)=(p^0\star \Phi)_t(x,y)+(r\star \Phi)_t(x,y).
$$
The formal solution to this equation is given by the convolution
\be\label{r}
r_t(x,y)=(p^0\star \Psi)_t(x,y),
\ee
where $\Psi$ is the sum of $\star$-convolutions of $\Phi$:
\be\label{Psi}
\Psi_t(x,y):=\sum_{k\geq 1}\Phi^{\star k}_t(x,y).
\ee
If  the series (\ref{Psi})  converges and the convolution  (\ref{r}) is well defined, we obtain   the required function $p_t(x,y)$ in the form
\be\label{sol_1}
p_t(x,y)=p_t^0(x,y)+\sum_{k\geq 1}(p^0\star\Phi^{\star k})_t(x,y).
\ee

Clearly, the above argument is yet purely formal; in order  to make it rigorous, we need to prove that  the parametrix construction is feasible, i.e. that the sum in the right hand side  of (\ref{sol_1}) is well defined, and then to associate  $p_t(x,y)$  with the initial operator $L$. We note that the key point to make the entire approach successful is the proper choice of the zero order approximation $p^0_t(x,y)$; see Section~\ref{over} for more detailed discussion of this point.

\subsection{The main results}\label{main}

Our standing assumption on the intensity coefficient $a(x)$ is that it is \emph{strictly positive, } \emph{bounded}  from above and below  and  \emph{H\"older continuous}  with some index $0<\eta\leq 1$, i.e.  there exist $0<c<C$  such that
\begin{equation}\label{a_bdd_Hol}
 c\leq a(x)\leq C, \quad |a(x)-a(y)|\leq C|x-y|^\eta, \quad x,y\in \Re^d.
\end{equation}
We also assume that the drift coefficient $b(x)$  satisfies the assumption below:

\begin{equation}\label{b-bdd-Hol}
b(\cdot)\in C_b(\rd), \quad |b(x)-b(y)|\leq C |x-y|^\gamma, \quad x,y\in \Re^d,
\end{equation}
where $\gamma\in [0,1]$. We consider three cases; in each of them the H\"older index $\gamma$ for the drift coefficient is related to the  index $\alpha$.

\medskip

{\it Case} \textbf{A}.  $\alpha\in (1,2)$, $\gamma=0$.

{\it Case} \textbf{B}. $\alpha\in
((1+\gamma)^{-1}, 2)$,  $0<\gamma<1$.

{\it Case} \textbf{C}. $\alpha\in (0,2)$, $\gamma=1$.

\medskip



\begin{rem}\label{r22}  a) Observe that  assumptions in the cases  \textbf{A} -- \textbf{C}  overlap, but none of the assumptions is implied by the other one: if we denote by  $\alpha_A,\alpha_B, \alpha_C$  the infima of $\alpha$ allowed in each of these cases, then we have   $\alpha_A=1$, $\alpha_B=(1+\gamma)^{-1}$,  and $\alpha_C=0$.    Note that  the regularity assumption on drift coefficient is weakened  from case \textbf{A} to case \textbf{C}, but on the other hand    we have $\alpha_C\leq\alpha_B \leq \alpha_A$.  \normal Heuristically, this means that  by increasing the regularity of $b$ we can relax the assumption on $\alpha$, and vice versa.
Note also that if we let  $\gamma\to 1$ in the ``intermediate case'' \textbf{B}, we get $\alpha_B\to 1/2\not=\alpha_C$, which means that case \textbf{C} cannot be obtained from \textbf{B} by such a  limit procedure.

b) As we will see below (cf. Proposition~\ref{p1}), the constructed function  $p_t(x,y)$ is uniquely  associated with the operator in  \eqref{symbol}, which implies  that if the coefficients $a(x)$ and $b(x)$ are such that  some of the cases \textbf{A} -- \textbf{C} overlap,  then different  choices of the  zero-order approximation provided by \eqref{abc} below  give the same outcome  $p_t(x,y)$.  However, the upper and lower bounds on $p_t(x,y)$ depend on the choice of $p^0_t(x,y)$. \normal
\end{rem}

When  $b$ is Lipschitz continuous, the Cauchy problem for the ordinary differential equation (ODE)
\begin{equation}\label{flow1}
    d\chi_t=b(\chi_t)\, dt, \quad \chi_0=x, \normal
\end{equation}
admits  the \emph{flow of solutions} $\{\chi_t, t\in \Re\}$, cf. \cite[Th.~2.1]{CL55}.    Denote by $\{\theta_t=\chi_t^{-1}, t\in \Re\}$ the \emph{inverse flow}, i.e. $\big(\theta_t\circ \chi_t \big) (x)= \big(\chi_t\circ \theta_t \big) (x)=x$,  which can be also defined just as the flow of solutions to the  Cauchy problem for the ODE
\begin{equation}\label{flow2}
   d\theta_t=-b(\theta_t)\, dt, \quad \theta_0=x.\normal
\end{equation}

In all the results formulated in the sequel, we assume that \eqref{a_bdd_Hol}, \eqref{b-bdd-Hol} hold true, and one of three assumptions which relate $\alpha$ and $\gamma$ (cases \textbf{A} -- \textbf{C}) is satisfied.
In our first main result we specify  in each of the cases \textbf{A} -- \textbf{C}  the choice of the zero order approximation $p^0_t(x,y)$ in the parametrix construction outlined above, and prove that this construction is feasible.

\begin{thm}\label{t1} Let
\begin{equation}\label{p0}
p_t^0(x,y):= \frac{1}{t^{d/\alpha}a^{d/\alpha}(y)}g^{(\alpha)}\left({\omega(t,y)-x\over t^{1/\alpha}a^{1/\alpha}(y)}\right),
\end{equation}
where
\begin{equation}\label{abc}
\omega(t,y):=
\begin{cases}
y, & \text{ in case \textbf{A};}\\
y-tb(y),&  \text{ in case \textbf{B};}\\
\theta_t(y), &\text{ in case \textbf{C}}.
\end{cases}
\end{equation}

Then  the following statements hold true.

 \begin{itemize}
 \item[1.]  For $t>0$, $x,y\in \rd$, the function $p_t(x,y)$ given by  \eqref{sol_1} is  well defined, in the sense that the integrals $\Phi^{\star k}$ and $p^0 \star \Phi^{\star k}$ exist, and for every $T>0$ the series involved in \eqref{sol_1} converges absolutely on $(0,T]\times \rd\times \rd$,  i.e.
     $$
     \sum_{k\geq 1}\big|(p^0\star\Phi^{\star k})_t(x,y)\big|<\infty.
     $$

\item[2.]  The function $p_t(x,y)$ is continuous on $(0, \infty)\times \Re^d\times\Re^d$.

\item[3.] For any  $\kappa\in (0,\alpha\wedge \eta)$ and $T>0$, the following estimate for  $r_t(x,y)$  holds true:
\begin{equation}\label{r_bound}
|r_t(x,y)|\leq Cp_t^0(x,y) V_t\Big(\omega(t,y)-x\Big),\quad t\in (0,T], \, x,y\in \rd,
\end{equation}
where
\begin{equation}\label{r_bound2}
V_t(z)=
\begin{cases}
t^{\kappa/\alpha}+t^\delta,& \text{if}\quad |z|\leq t^{1/\alpha}\\
|z|^{\kappa} +  t^\delta,&\text{if}\quad t^{1/\alpha}\leq |z|\leq 1,\\
1+t^\delta, &\text{if}\quad |z|\geq 1.
\end{cases}
\end{equation}
and
\begin{equation*}
\delta:=
\begin{cases}
\kappa, & \text{ in cases \textbf{A} and \textbf{C}}, \\
\left(1-{1\over \alpha}+\gamma\right)\wedge \left(1-{1\over \alpha}+{\gamma\over \alpha}\right)\wedge \kappa, & \text{ in case \textbf{B}.}
\end{cases}
\end{equation*}

\end{itemize}
\end{thm}
\begin{rem} We have
\be\label{H_0}
p_t^{(0)}(x,y) \asymp \frac{1}{t^{d/\alpha}} \frac{1}{(1+t^{-1/\alpha}|\omega(t,y)-x|)^{d+\alpha}}
\ee
(see \eqref{g_a} below). Hence  \eqref{r_bound} is equivalent to the following:
\begin{equation}\label{r_bound0}
|r_t(x,y)|\leq R_t\Big(\omega(t,y)-x\Big),\quad t\in (0,T], \, x,y\in \rd,
\end{equation}
where
\begin{equation}\label{r_bound20}
R_t(z)=
\begin{cases}
 ct^{-d/\alpha} ( t^{\kappa/\alpha}+ t^\delta), &  \text{if}\quad  |z|\leq t^{1/\alpha},\\
c\big(|z|^\kappa+t^\delta\big) \dfrac{t}{|z|^{d+\alpha}}, & \text{if}\quad t^{1/\alpha}<|z|\leq 1, \\
\dfrac{ct}{|z|^{d+\alpha}}, &  \text{if}\quad  |z|> 1.
\end{cases}
\end{equation}
\end{rem}

The proof of Theorem \ref{t1} and all the other   results in this section  formulated below are postponed to subsequent sections.  All the subsequent results use the same notation.

Next, we relate the  function $p_t(x,y)$ to the initial operator $L$. To make the structure of the proofs the most transparent, we do this in two steps.  First, we prove that $p_t(x,y)$ is a  transition probability density of \emph{some} Markov process, and that the $C_\infty(\Re^d)$-generator of the respective semigroup  is an extension of $(L, C^2_\infty(\Re^d))$.

\begin{thm}\label{t2}  Identity
 \be\label{semi}
P_tf(x)=\int_{\Re^d}p_t(x,y)f(y)\, dy, \quad f\in C_\infty(\Re^d),
\ee
 defines a strongly continuous conservative  contraction  semigroup of non-negative operators  on $C_\infty(\Re^d)$, which in turn defines a (strong) Feller Markov process $X$. Every function $f\in C^2_\infty(\Re^d)$ belongs to the domain $D(A)$ of the generator $A$ of this semigroup, and
 $$
 Af(x)=Lf(x)=a(x) L^{(\alpha)}f(x)+\Big(b(x), \nabla f(x)\Big),\quad  f\in C^2_\infty(\Re^d);
 $$
 that is, $(A,D(A))$ is an extension of  $(L, C^2_\infty(\Re^d))$.
\end{thm}

In the next theorem  we prove that the semigroup (\ref{semi}) with $p_t(x,y)$ defined by (\ref{sol_1}) is in fact the  \emph{unique} Feller semigroup  associated with the operator (\ref{symbol}). 
\begin{thm}\label{t3}
The generator $(A,D(A))$ is the closure of $(L, C^2_\infty(\Re^d))$  in $C_\infty(\rd)$.
 \end{thm}

The proposition below clarifies  the relation between the function $p_t(x,y)$ and the notion of the ``fundamental solution'', on which the parametrix construction of $p_t(x,y)$ was based. We formulate and prove this proposition  under the additional assumption  that in the case \textbf{C} the function $b(\cdot)$ is continuously differentiable (no additional assumptions in the cases \textbf{A},\textbf{B} are required).

\begin{prop}\label{p1} The  real-valued function $p_t(x,y)$  is a fundamental solution to the Cauchy problem  for the operator $\partial_t - A$; that is, for  $t>0$ it is differentiable in  $t$, belongs to the domain of $A$ as a function of $x$, and satisfies (\ref{L_delta}) and  the analogue of (\ref{L_fund}) with $A_x$ instead of $L_x$.

\end{prop}

Our next step is to relate the process $X$ constructed in Theorem~\ref{t2} to the weak solution to the SDE driven by $\alpha$-stable noise or, in a closely related terminology, to the solution of the martingale problem for $(L, C^2_\infty(\Re^d))$. Namely, the semigroup $\{P_t\}$  corresponding to the  process $X$ possesses the Feller property, hence the process $X$ has a c\'adl\'ag modification, see  \cite[Chap.~4, Th.~2.7]{EK86}.  Denote by $\P_x$ the law of the Markov process $X$ with $X_0=x$ in the Skorokhod space $\DD([0, \infty),\Re^d)$ of c\'adl\'ag functions $[0, \infty)\to \Re^d.$ Recall that a measure $\P$ on $\DD([0, \infty),\Re^d)$ is called a solution to the \emph{martingale problem}  $(L, D(L))$, if for every $f\in D(L)$ the process
$$
f(X_t)-\int_0^t Lf(X_s)\,ds, \quad t\geq 0
$$
is a martingale with respect to $\P$, and the martingale  problem for   $(L, D(L))$ is called \emph{well posed}, if for every $x\in \Re^d$ there exists the unique solution to $(L, D(L))$ with $\P(X_0=x)=1$.

\begin{thm}\label{tweak} For    every     $x\in \Re^d$ the SDE (\ref{SDE}) with the initial condition
$X_0=x$ has a unique weak solution, and the law of this solution in $\DD([0, \infty),\Re^d)$ equals $\P_x$.

In addition, the martingale problem $(L, C_\infty^2(\rd))$ is well posed, and $\P_x$ is  its unique solution with the initial condition $X_0=x$.
\end{thm}
\begin{rem}\label{strong}
It is well known  that, when both coefficients in equation of \eqref{SDE} are Lipschitz continuous,  there exists a unique strong solution to \eqref{SDE} (see, for example, \cite[Th.~IV.9.1]{IW81}, or \cite[Th.~IV.3]{GS82}). In our framework, the coefficient $a(x)$ is assumed to be only H\"older continuous. Up to our knowledge, there are no results on the existence and uniqueness of the strong solution under the assumption of the H\"older continuity of coefficients in equations of type \eqref{SDE}, see also  \cite{BBC04} for the negative example.  See, however, \cite{Pr14},  \cite{CSZ15}, where under the assumption that $a=1$ and $b$ is H\"older continuous the existence and uniqueness of the strong solution is shown.
\end{rem}
The last two theorems  contain  explicit estimates, respectively, for  $p_t(x,y)$ and for its derivative with respect to  the time variable.

\begin{thm}\label{t4} We have
\begin{equation}\label{uplo}
p_t(x,y)\asymp
 \frac{1}{t^{d/\alpha}}g^{(\alpha)}\left({\omega(t,y)-x\over t^{1/\alpha}}\right), \quad t\in (0,T], \, x,y\in \rd.
\end{equation}
\end{thm}

Finally, in the theorem below we show the  continuity of the time derivative $\prt_t p_t(x,y)$, and provide the upper estimate for it. Note that these properties of $\prt_t p_t(x,y)$ are involved in the proofs of Theorem~\ref{t3} and Proposition~\ref{p1}; see assertions \eqref{dPte}, \eqref{dpte} below. These properties also have an independent interest, e.g. in the context of estimation of the accuracy of discrete approximation of occupation time functionals; see \cite{GK14}.
\begin{thm}\label{t5}

 \begin{enumerate}
   \item There exists a set $\Upsilon\subset (0, \infty)\times \Re^d$ of zero Lebesgue measure such that the function $p_t(x,y)$ defined by (\ref{sol}) -- (\ref{Psi})  has a derivative
$$
\prt_tp_t(x,y), \quad x\in \Re^d,\quad (t,y) \notin \Upsilon,
$$
which for every fixed $(t,y) \notin\Upsilon$ is continuous in  $x$. Moreover, in the cases \textbf{A} and  \textbf{B} the set $\Upsilon$ is empty, and $\prt_tp_t(x,y)$ is continuous in $(t,x,y)$.
   \item The derivative $\prt_tp_t(x,y)$ possesses the bound
   $$
   |\prt_tp_t(x,y)|\leq C\Big(t^{-1}\vee t^{-1/\alpha}\Big) \frac{1}{t^{d/\alpha}}g^{(\alpha)}\left({\omega(t,y)-x\over t^{1/\alpha}}\right), \quad x\in \Re^d,\quad (t,y) \notin \Upsilon.
   $$

   \item For every $f\in C_\infty(\rd)$  the function
   $$
   (0,\infty)\ni t\mapsto P_tf\in C_\infty(\rd)
   $$
   is continuously differentiable, and
   $$
   (\prt_tP_tf)(x)=\int_\rd \prt_tp_t(x,y) f(y)dy.
   $$
    \end{enumerate}

\end{thm}

\begin{rem} For $\lambda >  0$ denote
\begin{equation}\label{Gbet}
G^{(\lambda)}(x) := \big( |x|\vee 1)^{-d-\lambda}, \quad x\in\rd.
\end{equation}
   Since     $g^{(\alpha)}(x)\asymp G^{(\alpha)}(x)$ (see Proposition~\ref{A2} below), one can replace in the above bounds $g^{(\alpha)}$ by $G^{(\alpha)}$, which gives  more explicit estimates. We   used $g^{(\alpha)}$  in the estimates for $p_t(x,y)$ and $\partial_t p_t(x,y)$  in order to emphasize the impact of
 the original $\alpha$-stable process.
 \end{rem}

\subsection{Overview and discussion}\label{over}

    For the description and the background of the parametrix construction of the fundamental solution to a Cauchy problem for  parabolic second  order  PDE's, we refer to the monograph of Friedman \cite{Fr64}; see also the original papers by E.Levi \cite{Le07} and W. Feller \cite{Fe36}. This construction was  extended  in \cite{Dr77},  \cite{ED81}, \cite{Ko89},  \cite{Ko00}, to equations with pseudo-differential operators, see also the reference list and an  extensive overview in the monograph \cite{EIK04}. In \cite{Dr77},  \cite{ED81}, \cite{Ko89}, the ``main term'' in the pseudo-differential operator is assumed to have the form $a(x)L^{(\alpha)}$ (in our notation) with $\alpha>1$. In \cite{Ko00} the operator of such    a   type is treated, and  although   the case   $\alpha\leq 1$ is allowed,  in this case the gradient term is not  involved in the equation.  The list of subsequent and related  publications is large, and we cannot discuss it here  in details. Let us only mention two recent preprints:   \cite{CZ13},  where two-sided estimates, more precise than those in \cite{Ko89} are obtained,  and \cite{BK14},  where the  probabilistic interpretation of the parametrix construction and its application to the Monte-Carlo  simulation is developed.

    In all the references listed above it is required that either  the stability index $\alpha$ satisfies   $\alpha>1$, or the gradient term is not  involved in the equation. This is the common assumption in all the references available for us in this direction, with the one important exception given by the recent paper \cite{DF13}, see also \cite{FP10}, \cite{F13}. In  \cite{DF13}, for a L\'evy driven SDE with $\alpha$-stable like noise, the question of \emph{existence} of the distribution density is studied by a different method, based on thoroughly balanced  approximation of the initial SDE, Fourier transform based estimates, and discrete integration by parts. \normal Such an approach is applicable for  SDE's  with  the    parameter of the noise  $\alpha<1$ \normal  and  a  non-trivial  drift, but it does not give proper tools neither for obtaining explicit estimates for this density, nor even for proving  the existence and uniqueness of the solution to the initial SDE. Hence, the scope of our approach based on the parametrix construction, differs substantially from that  of  \cite{DF13}.

    Our version of  the parametrix construction contains a substantial novelty, which makes it possible  to handle the case  $\alpha\leq 1$ with non-trivial  gradient term. To explain this modified construction  in the most transparent  way, we took the ``jump component'' in a relatively simple form $a(x)Z^{(\alpha)}$. Clearly,  one can think about considering, for example,  $\alpha$-stable symbol  with state dependent spectral measure, see \cite{Ko00}. This, however, leads to additional cumbersome but inessential technicalities,  and we prefer not do this in the current paper.

 It was already mentioned that the parametrix construction described in Section~\ref{sA} heavily  relies on the choice of the ``zero-order''  approximation $p_t^0(x,y)$. In the case $\alpha>1$,  the first term $a(x) L^{(\alpha)}$  dominates the second term  $b(x) \nabla $ in the sense that the symbol $a(x) |\xi|^\alpha$ of the first term  grows as  $|\xi|\to \infty$ faster than the modulus of the symbol $ib(x)\xi$ of the second term, see \cite{Ja02} for  the detailed explanation. This allows  us to chose  in the case $\alpha>1$  the zero-order approximation $p^0$  in the ``classical''  way  (cf. \cite{Fr64},  \cite{Ko89}):  Take  the ``principal part'' $a(x) L^{(\alpha)}$ of the generator,  ``freeze'' the coefficient $a(x)$ at some point $x=z$,  then take the fundamental solution $q^z_t(x,y)$ to the operator $\partial_t- a(z)L^{(\alpha)}$ with this ``frozen principal symbol'', and finally put $p^0_t(x,y):= q^z_t(x,y)|_{z=y}$. However, this procedure is not successful in the case $\alpha\leq 1$, and the reason for this is already mentioned: in this case, the operator $a(x) L^{(\alpha)}$ no longer dominates the gradient term, and therefore it can not be treated as  the ``principal part'' of $L$. One can modify the choice of the zero order approximation in several ways.   One of the possible choices\normal \,  is to ``freeze'' the coefficients in the entire operator $L$ (which is now itself  considered as   the   ``principal part''),  and   to   take the fundamental solution $q^z_t(x,y)$ to the Cauchy problem for $\partial_t- a(z)L^{(\alpha)}-(b(z),\nabla)$. This is exactly what we do in  Case \textbf{B}. However, this procedure is   restricted     by the relation between the parameter of the H\"older continuity $\gamma$ and $\alpha$.  In  Case \textbf{C} the choice of the zero order approximation is no longer related   to   the fundamental solution to an  equation with frozen coefficients, but instead uses a carefully designed ``corrector'' $\omega(t,y)$, which   allows us\normal \,  to treat all  $\alpha\in (0,2)$.  The price of such an approach is the assumption of the Lipschiz continuity of the drift $b$.
\normal

The effect of the interplay between the value of $\alpha$ and the regularity properties required for $b$, which we have mentioned in  Remark \ref{r22}, was   observed first in \cite{Po94}, \cite{PP95}. It was shown therein that the parametrix construction is still feasible for  (possibly unbounded) $b\in L_{p}(\Re^d)$, $p>d/(\alpha-1)$, where $\alpha>1$. In \cite{BJ07} this effect was rediscovered in a stronger form:  it is required  that $b$ belongs to the Kato class $\mathbb{K}_{d,\alpha-1}$. In \cite{KS14} this result is extended even further, with $b$ being allowed to be a generalized function equal to the derivative of a measure from $\mathbb{K}_{d,\alpha-1}$.

    In general, there is a substantial gap between the problem of  constructing  a  ``candidate for being the fundamental solution'' (i.e., to prove  that relations (\ref{sol}) -- (\ref{Psi}) make sense), and the problem of relating the constructed  kernel $p_t(x,y)$ to a Markov process.   The first way how one can    possibly solve  this problem was proposed in  \cite{Ko89}. It extends   the approach   from  \cite{Fr64}, where it is shown in the parabolic setting
    that $p_t(x,y)$ is twice continuously differentiable in  $x$,  and satisfies (\ref{L_fund}) in  the classical sense. Note that the domain of  the operator $L$  is $C_b^2(\rd)$.   In the $\alpha$-stable case the natural upper bound $\prt_{xx}^2p_t^0(x,y)\leq Ct^{-2/\alpha}$ is strongly singular for small $t$, and therefore it is difficult to prove using the parametrix construction that $p_t(\cdot ,y)\in C_b^2(\rd)$; thus, one cannot check  straightforwardly  that  the expression \eqref{L_fund} makes sense.  Instead, in \cite{Ko89} the extended domain of $L^{(\alpha)}$ is  introduced  in  terms of ``hyper-singular integrals'', and it is proved that $p_t(x,y)$ satisfies (\ref{L_fund}) in the corresponding ``extended'' sense. Once (\ref{L_fund}) is proved, the required properties of the Markov process associated with $p_t(x,y)$ follow from the positive maximum principle in a rather standard way. Another way to   verify\normal \, the parametrix construction, proposed in \cite{Ko00}, is to guarantee the required smoothness of $p_t(x,y)$ by using the integration by parts procedure, but this approach seems to be only partially relevant; see Remark~\ref{r_Kol} below.

    Partially, one  can solve the  problem of relating the kernel $p_t(x,y)$ to a Markov process  by using some  approximation procedure (e.g. \cite{Po94}, \cite{PP95}), or by analysing the perturbation of the resolvent kernels (cf. \cite{BJ07}). However, the most difficult part  here is to relate \emph{uniquely} the initial symbol and  the Markov process associated with $p_t(x,y)$. This problem was  recently  solved in \cite{KS14}, in the framework of a singular gradient perturbation of an $\alpha$-stable generator, in  terms of the weak solution to the corresponding SDE. See also \cite{CW13}, where alternatively the martingale problem approach was used. The technique therein is closely related to those  introduced (in the diffusive setting) in \cite{BC03}, and apparently strongly relies on the structural assumption that the resolvent, which corresponds to $p_t(x,y)$, is a perturbation of the  resolvent of  an  $\alpha$-stable process.

    We propose a new method for solving this correspondence problem.   Our    method  is based  on the notion of the \emph{approximate fundamental solution} to the Cauchy problem for (\ref{L_ful}), see Section \ref{s4} and especially  the discussion at the beginning of Section \ref{s52}. Combined with a proper ``approximate'' version of the positive maximum principle, this notion gives a flexible tool both for proving  the semigroup properties of $p_t(x,y)$ (Theorem \ref{t2}), and for studying  more delicate uniqueness issues (Theorem \ref{t3}).
    We expect that this method  will be well applicable in other situations, where the parametrix construction is feasible; this is the subject of our ongoing research.

   Let us briefly discuss  another large group of results, focused on the construction of a \emph{semigroup} for a Markov process with a given symbol rather than on the transition probability density $p_t(x,y)$ for it.   An approach based on properties of the symbol of the operator and on the Hilbert space methods, is developed in the works of Jacob \cite{Ja94}, see also the monograph \cite{Ja96}. It allows to show the existence  of the closed extension  in $C_\infty(\rd)$ of a  given pseudo-differential operator, and that this extension is the generator of a Feller semigroup.  This approach was further developed   in      \cite{Ho98a}, \cite{Ho98b},  \cite{Bo05}, \cite{Bo08},   and relies on the symbolic calculus approach for the  parametrix  construction    (cf.  \cite{Ku81},  the original papers  \cite{Ts74}, \cite{Iw77}),  and see also   \cite{Ja01}--\cite{Ja05} for the detailed treatment).

Finally, we mention the  group of results devoted to the well-posedness of the martingale problem for an integro-differential operator of certain type. For different types of perturbations of an $\alpha$-stable generator, this problem was treated in   \cite{Ko84a}, \cite{Ko84b},  \cite{Ts70}, \cite{Ts73}, \cite{MP92a}--\cite{MP12a},     \cite{Ba88},
see also     \cite{Ho94}, \cite{Ho95}  for yet another approach  for rather wide class of operators.

\section{Proof of Theorem~\ref{t1} and continuity properties of $P_t$} \label{s2}

\subsection{Function $\Phi$: evaluation and an upper bound}\label{32}
Our first step in the proof of Theorem~\ref{t1} is to evaluate the kernel $\Phi$  and to give an upper bound for it.

For $\lambda\in [0,\alpha)$ we introduce a family of kernels
\begin{equation}\label{Qt21}
Q_t^{(\lambda)}(x,y):=\left( \Big|\frac{\omega(t,y)-x}{t^{1/\alpha}}\Big|^\lambda  \wedge t^{-\lambda/\alpha} \right)\frac{1}{t^{d/\alpha}}   G^{(\alpha)}\normal \left({\omega(t,y)-x\over t^{1/\alpha}}\right),
\end{equation}
  where the function $G^{(\alpha)}(x)$ in defined in \eqref{Gbet}, and $\omega(t,y)$ in defined in \eqref{abc}\normal.
We remark that since  $Q_t^{(\lambda)} (x,y)$ involves $\omega(t,y)$,  we have in fact  three different families  $Q_t^{(\lambda)} (x,y), \lambda\in [0,\alpha)$, which correspond to
 the cases \textbf{A} -- \textbf{C}.

\begin{lem}\label{lPhi1}
Let $\kappa\in (0,\alpha\wedge \eta), T>0$.  Then
\begin{equation}\label{Phi1}
|\Phi_t(x,y)|\leq C\Big( t^{-1+\kappa/\alpha}Q_t^{(\kappa)} (x,y)+ t^{-1+\delta} Q_t^{(0)}(x,y)\Big), \quad t\in (0,T], \quad x,y\in \Re^d,
\end{equation}
where $\delta=\kappa$ in    the    cases \textbf{A} and \textbf{C}, and    $\delta= \left(1-1/ \alpha+\gamma\right)\wedge \left(1-1/\alpha+\gamma/\alpha\right)\wedge \kappa$ in the     case \textbf{B}.
\end{lem}
 Before we proceed to the proof, we  formulate  some auxiliary statements.
\begin{prop}\label{A1}
\begin{enumerate}
  \item For any   $\lambda>0$, $c>0$ there exists $C>0$ such that
\begin{equation}\label{G1}
G^{(\lambda)}(cx)\leq C G^{(\lambda)}(x).
\end{equation}
  \item For any  $\lambda_1>\lambda_2$ we have
\begin{equation}\label{G2}
 G^{(\lambda_1)}(x)\leq  G^{(\lambda_2)}(x).
\end{equation}
  \item For any $\eps\in (0, \lambda)$,
\begin{equation}\label{G3}
|x|^{\eps} G^{(\lambda)}(x)\leq C G^{(\lambda-\eps)}(x).
\end{equation}
\end{enumerate}
\end{prop}
  The proof of Proposition~\ref{A1} is obvious; we omit the details. \normal

\begin{prop}\label{A2}  For any $\alpha\in (0,2)$,
\begin{equation}\label{g_a}
g^{(\alpha)}(x)\asymp G^{(\alpha)}(x),
\end{equation}
\begin{equation}\label{g_a_der}
\Big|(\nabla g^{(\alpha)})(x)\Big|\leq C G^{(\alpha+1)}(x),
\end{equation}
\begin{equation}\label{g_a_der2}
\Big|(\nabla^2 g^{(\alpha)})(x)\Big|\leq C G^{(\alpha+2)}(x),
\end{equation}
\begin{equation}\label{g_a_frac}
\Big|(L^{(\alpha)} g^{(\alpha)})(x)\Big|\leq C G^{(\alpha)}(x),
\end{equation}
\begin{equation}\label{g_a_frac_der}
\Big|(\nabla L^{(\alpha)} g^{(\alpha)})(x)\Big|\leq C G^{(\alpha+1)}(x).
\end{equation}

\end{prop}
  The results stated in Proposition~\ref{A2} are partly known; we defer the discussion and the remaining proofs to Appendix~A.\normal

\begin{proof}[Proof of Lemma~\ref{lPhi1}] We consider    the     cases \textbf{A} -- \textbf{C} separately. To improve the readability, here and below we assume that $T>0$ is fixed and, if it is not stated otherwise, in      every    formula containing $t,x$, or $y$  we assume $t\in (0,T], x\in \Re^d, y\in \Re^d$.

\emph{Case}  \textbf{A}. Fix $z\in \Re^d$, and denote
$$
L^z=a(z)L^{(\alpha)};
$$
that is, consider ``the principal part'' of the operator $L$ with the coefficient ``frozen'' at the point $z$ (cf. the discussion in Section~\ref{over}). Denote by $q^z_t(x,y)$  the transition probability density of the process $Z^{(\alpha)}$ with the time, re-scaled by $a(z)$:
$$
  q_t^z(x,y)={1\over t^{d/\alpha}a^{d/\alpha}(z)}g^{(\alpha)}\left({y-x\over t^{1/\alpha}a^{1/\alpha}(z)}\right).
$$
Then $q^z_t(x,y)$  is a fundamental solution to the Cauchy problem for the operator $(\prt_t-L^z)$,
and  $$
p_t^0(x,y)=q_t^y(x,y).
$$

Observe that for     every \normal  fixed $x,y\in \Re^d$ the function $p_t^0(x,y)$   belongs to $C^1((0,\infty))$ as a function of $t$,
  and for    every \normal   fixed $t\in (0, \infty)$ and  $y\in \Re^d$ it  belongs to $C^2_\infty(\Re^d)$ as a function of $x$.
 Since both $\nabla$ and $L^{(\alpha)}$ are well defined on   $C_\infty^2(\Re^d)$, the function $\Phi_t(x,y)$ is well defined by (\ref{Phi}).

  Since $q^z$ is a fundamental solution for $\prt_t-L^z$, one has
\be\label{Phi_a}\ba
\Phi_t(x,y)&=\left[-\Big(\prt_t-L^z_x\Big)p_t^0(x,y)+(L_x-L^z_x)p_t^0(x,y)\right] \Big|_{z=y}=(L_x-L^z_x)p_t^0(x,y) \Big|_{z=y}\\
&= \Big(a(x)-a(y)\Big) L^{(\alpha)}_xp_t^0(x,y)+\Big(b(x), \nabla_x p_t^0(x,y)\Big)\\
&= \Big(a(x)-a(y)\Big)  {1\over t^{d/\alpha+1}a^{d/\alpha+1}(y)}(L^{(\alpha)} g^{(\alpha)})\left({y-x\over t^{1/\alpha}a^{1/\alpha}(y)}\right)\\
&\hspace*{0.5cm}-
{1\over t^{(d+1)/\alpha}a^{(d+1)/\alpha}(y)}\left(b(x), (\nabla g^{(\alpha)})\left({y-x\over t^{1/\alpha}a^{1/\alpha}(y)}\right)\right)=:\Phi_t^1(x,y)+\Phi_t^2(x,y).
\ea
\ee

First we estimate $\Phi^1$.   Since $a(x)$ is bounded    from above and away from zero\normal, and is $\eta$-H\"older continuous, we have
\begin{equation}\label{a_hol_in}
|a(x)-a(y)|\leq c\big(|x-y|^\eta\wedge 1\big)\leq c\big(|x-y|^\kappa\wedge 1\big),
\end{equation}
 where we used that $\kappa<\alpha\wedge \eta$ (recall that by $c$ and  $C$ we denote the generic constants, which may vary from place to place).
Then  by (\ref{g_a_frac}), \eqref{G1}   and \eqref{Qt21}    we obtain 
\begin{equation}\label{esP1a}
\begin{split}
|\Phi_t^1(x,y)|& \leq C \frac{|y-x|^\kappa\wedge 1}{t^{1+d/\alpha}} G^{(\alpha)} \left({y-x\over t^{1/\alpha}}\right)
\leq  C  t^{-1+\kappa/\alpha} Q_t^{(\kappa)}(x,y).
\end{split}
\end{equation}
To estimate $\Phi^2$    we use that     the functions  $a(x)$ and $ b(x)$ are  bounded, and $a(x)$ is bounded away from zero. Hence by  (\ref{g_a_der}),  \eqref{G1}  and \eqref{G2}  we have
\begin{equation}\label{esP2a}
|\Phi_t^2(x,y)|\leq Ct^{-(d+1)/\alpha}G^{(\alpha+1)}\left({y-x\over t^{1/\alpha}}\right)
\leq Ct^{-1/\alpha}Q_t^{(0)}\left({y-x\over t^{1/\alpha}}\right).
\end{equation}
Combining estimates \eqref{esP1a} and \eqref{esP2a}, we
 obtain the required estimate.

\medskip

\emph{Case}  \textbf{B}.   We fix $z\in \Re^d$, and define
$$
L^z=a(z)L^{(\alpha)}+\Big(b(z), \nabla \Big);
$$
that is, consider the entire operator $L$ as its ``principal part'' and ``freeze'' its coefficients at the point $z$ (cf.  the discussion in Section~\ref{over} and the  proof in    the   case \textbf{A}).  The fundamental solution $q^z_t(x,y)$ to the Cauchy problem for  $(\prt_t- L^z)$  is equal  to the transition probability density of the process $Z^{(\alpha)}$ with the time parameter  rescaled by $a(z)$, and with the additional  constant drift $tb(z)$:
\begin{equation}\label{qxi2}
   q_t^z(x,y)={1\over t^{d/\alpha}a^{d/\alpha}(z)}g^{(\alpha)}\left({y-x-tb(z)\over t^{1/\alpha}a^{1/\alpha}(z)}\right).
\end{equation}
Again, we have
$$
p_t^0(x,y)= q_t^y(x,y).
$$
Since $q^\xi$ is the fundamental solution for $\prt_t-L^\xi$, we have in the same way as in  (\ref{Phi_a})
\be\label{Phi_b}\ba
\Phi_t(x,y)&=\left[-\Big(\prt_t- L^z_x\Big)p_t^0(x,y)+(L_x-L^z_x)p_t^0(x,y)\right]\Big|_{z=y}\\
&= \Big(a(x)-a(y)\Big)  L^{(\alpha)}_xp_t^0(x,y)+\Big(b(x)-b(y), \nabla_xp_t^0(x,y)\Big)\\
&= \Big(a(x)-a(y)\Big)  {1\over t^{d/\alpha+1}a^{d/\alpha+1}(y)}(L^{(\alpha)} g^{(\alpha)})\left({y-tb(y)-x\over t^{1/\alpha}a^{1/\alpha}(y)}\right)\\&\hspace*{0.5cm}+
{1\over t^{(d+1)/\alpha}a^{(d+1)/\alpha}(y)}\left(b(y)-b(x), (\nabla g^{(\alpha)})\left({y-tb(y)-x\over t^{1/\alpha}a^{1/\alpha}(y)}\right)\right)\\
&=:\Phi_t^1(x,y)+\Phi_t^2(x,y).
\ea
\ee
Recall \eqref{a_hol_in}, and write
$$
x-y=(y-tb(y)-x)+tb(y).
$$
Then by  elementary   inequalities $|u+v|^{\kappa}\leq 2^{\kappa-1} (|u|^\kappa+ |v|^\kappa)$,  $|u+v|^\kappa \wedge 1 \leq |u|^\kappa \wedge 1 + |v|^\kappa \wedge 1 $,  and the fact that $b(\cdot)$ is bounded, we obtain
\begin{equation}\label{a_hol_in_b}
|a(x)-a(y)| \leq c\big(|y-tb(y)-x|^\kappa\wedge 1\big)+ct^\kappa.
\end{equation}
Then  using (\ref{g_a_frac}),  \eqref{G1}   and \eqref{Qt21}   we derive 
\begin{equation}\label{Phi_1_bound-b}
\begin{split}
|\Phi_t^1(x,y)|&\leq Ct^{-1+\kappa/\alpha}Q_t^{(\kappa)}(x,y)+Ct^{-1+\kappa}Q_t^{(0)}(x,y).
\end{split}
\end{equation}
Similar argument can be applied to $\Phi^2_t(x,y)$. Namely,  using that $b(\cdot)$ is $\gamma$-H\"older continuous and bounded, we get
$$
|b(x)-b(y)| \leq c|y-tb(y)-x|^\gamma+ct^\gamma.
$$
Then using \eqref{g_a_der} and  \eqref{G1}--\eqref{G3} we derive
\begin{equation}\label{Phi_2_bound-b}
\begin{split}
|\Phi_t^2(x,y)|&\leq Ct^{-1/\alpha+\gamma/\alpha} \Big|\frac{y-tb(y)-x}{t^{1/\alpha}} \Big|^{\gamma}  \frac{1}{t^{d/\alpha}}G^{(\alpha+1)}\left({y-tb(y)-x\over t^{1/\alpha}}\right)\\
&\quad +Ct^{-1/\alpha+\gamma}\frac{1}{t^{d/\alpha}} G^{(\alpha+1)}\left({y-tb(y)-x\over t^{1/\alpha}}\right)\\
&\leq C  t^{-1/\alpha+\gamma/\alpha} \frac{1}{t^{d/\alpha}}G^{(\alpha+1-\gamma)}\left({y-tb(y)-x\over t^{1/\alpha}}\right) +Ct^{-1/\alpha+\gamma}\frac{1}{t^{d/\alpha}} G^{(\alpha+1)}\left({y-tb(y)-x\over t^{1/\alpha}}\right) \\
&\leq Ct^{-1+\zeta}Q_t^{(0)}(x,y),
\end{split}
\end{equation}
where $\zeta:= \left(1-1/\alpha+\gamma\right)\wedge \left(1-1/\alpha+\gamma/\alpha\right)   >0\normal$. Thus, we  arrive at \eqref{Phi1}.

\medskip
\medskip

\emph{Case}  \textbf{C}.  In contrast with two previous cases, now  we cannot interpret $p_t^0(x,y)$ as a fundamental solution to  a Cauchy problem for some operator with ``frozen'' coefficients. Instead, we use the definition of the flow $\theta_t(y)$ and evaluate $\Phi$ directly.

Operators $\nabla$ and $L^{(\alpha)}$ are homogeneous with respective orders $1$ and  $\alpha$. From the identity
$$
(\partial_t - L^{(\alpha)})\big[t^{-d/\alpha} g^{(\alpha)}(t^{-1/\alpha}x)\Big] =0,
$$
we derive
\begin{align*}
\prt_tp_t^0(x,y)&=\Big[a(y) \frac{1}{a^{d/\alpha}(y)t^{d/\alpha} } (L^{(\alpha)} g^{(\alpha)})\left(\frac{w}{t^{1/\alpha}  a^{1/\alpha}(y)\normal} \right)\\
&\quad +\Big(\prt_t\theta_t(y), \frac{1}{a^{d/\alpha+1}(y)t^{d/\alpha+1} } (\nabla g^{(\alpha)})\left(\frac{w}{t^{1/\alpha}  a^{1/\alpha}(y)\normal} \right)\Big)\Big]\Big|_{w=\theta_t(y)-x}.
\end{align*}
On the other hand,
\begin{align*}
L_xp_t^0(x,y)&=\Big[a(x)\frac{1}{a^{d/\alpha}(y)t^{d/\alpha} } (L^{(\alpha)} g^{(\alpha)})\left(\frac{w}{t^{1/\alpha}  a^{1/\alpha}(y)\normal} \right)\\
&\quad \quad -\Big(b(x), \frac{1}{a^{d/\alpha+1}(y)t^{d/\alpha+1} } (\nabla g^{(\alpha)})\left(\frac{w}{t^{1/\alpha}  a^{1/\alpha}(y)\normal} \right)\Big)\Big]\Big|_{w=\theta_t(y)-x}.
\end{align*}
   Since   $\prt_t\theta_t(y)=-b(\theta_t(y))$, we  finally  get
\begin{equation}\label{Phi_c}
\begin{split}
\Phi_t(x,y)&=\big(  L_x-\partial_t\big) p_t^0(x,y)\\
&=\Big(a(x)-a(y)\Big) {1\over t^{d/\alpha+1}a^{d/\alpha+1}(y)}(L^{(\alpha)} g^{(\alpha)})\left({\theta_t(y)-x\over t^{1/\alpha}a^{1/\alpha}(y)}\right)\\&+
{1\over t^{(d+1)/\alpha}a^{(d+1)/\alpha}(y)}\left(b(\theta_t(y))-b(x), (\nabla g^{(\alpha)})\left({\theta_t(y)-x\over t^{1/\alpha}a^{1/\alpha}(y)}\right)\right)
\\&=:\Phi_t^1(x,y)+\Phi_t^2(x,y).
\end{split}
\end{equation}
Since $b$ is bounded, we have $|\theta_t(y)-y|\leq ct$. Similarly to (\ref{a_hol_in_b}) we have
\begin{equation}\label{a_hol_in_c}
|a(x)-a(y)|\leq  c(|\theta_t(y)-x|^\kappa\wedge 1)+c t^\kappa.
\end{equation}
   Then using \eqref{g_a_frac}, \eqref{G1} and  \eqref{G2}, we get
\begin{align*}
|\Phi_t^1(x,y)|&\leq Ct^{-1+\kappa/\alpha}   \Big(\Big| \frac{\theta_t(y)-x}{t^{1/\alpha}}\Big|^\kappa \wedge t^{-\kappa/\alpha}\Big)\normal \frac{1}{t^{d/\alpha}}G^{(\alpha)}\left({\theta_t(y)-x\over t^{1/\alpha}}\right) \\ &\quad +Ct^{-1+\kappa}\frac{1}{t^{d/\alpha}} G^{(\alpha)}\left({\theta_t(y)-x\over t^{1/\alpha}}\right)\\
&  =\normal  C  t^{-1+\kappa/\alpha} Q_t^{(\kappa)}(x,y)+ C t^{-1+\kappa} Q_t^{(0)}(x,y).
\end{align*}
For $\Phi^2$  we have, since $b(x)$ is Lipschitz continuous,
$$
|\Phi_t^2(x,y)|\leq C t^{-d/\alpha}\left|{\theta_t(y)-x\over t^{1/\alpha}}\right|
 \left|(\nabla g^{(\alpha)})\left({\theta_t(y)-x\over t^{1/\alpha}a^{1/\alpha}(y)}\right)\right|.
$$
Thus, using (\ref{g_a_der}) and  \eqref{G1}--(\ref{G3}), we derive
$$
|\Phi_t^2(x,y)|\leq  C Q_t^{(0)}(x,y).
$$
Combining the above estimates,  we arrive at \eqref{Phi1}.

\end{proof}

\subsection{Convergence of the parametric series and the estimate for the residue term}\label{s21}

 To estimate the  convolution powers $\Phi^{\star k}_t(x,y)$, $k\geq 1$, inductively we first slightly modify the upper bound for $\Phi$ obtained in Lemma~\ref{lPhi1}. For $\lambda\in [0,\alpha)$,  define
\begin{equation}\label{Ht21}
H_t^{(\lambda)} (x,y):=
\left(\left( \Big|\frac{\omega(t,y)-x}{t^{1/\alpha}}\Big|^\lambda \vee 1 \Big) \wedge t^{-\lambda/\alpha}\right) \right)\frac{1}{t^{d/\alpha}}  G^{(\alpha)}\normal \left({\omega(t,y)-x\over t^{1/\alpha}}\right).
\end{equation}
 Clearly,
 $$
 Q^{(\lambda)}_t(x,y)\leq H^{(\lambda)}_t(x,y),
 $$
 and therefore a (weaker) analogue of \eqref{Qt21} with  $Q^{(\kappa)}_t(x,y)$, $Q^{(0)}_t(x,y)$ replaced by $H^{(\kappa)}_t(x,y)$, $H^{(0)}_t(x,y)$ holds true. The reason for us to modify (in fact, to weaken) estimate    \eqref{Phi1}    in that way is that this form is well designed for  further inductive estimation of convolution powers of $\Phi$; see the  detailed discussion of this point in Remark~\ref{rem_conv} below. For possible  further reference, we first develop this calculation in a general form, and then apply it for the  particular function $\Phi_t(x,y)$ and kernels $H^{(\lambda)}_t(x,y)$.

\begin{dfn}  A non-negative kernel $\{H_{t}(x,y), t>0, x,y\in \Re^d\}$ has a sub-convolution property, if for every $T>0$ there exists a constant $C_{H,T}>0$ such that
\begin{equation}\label{H0}
(H_{t-s}* H_s)(x,y)\leq C_{H,T} H_{t}(x,y), \quad t\in (0, T], \quad s\in (0, t), \quad x,y\in \Re^d.
\end{equation}
The kernel $\{H_{t}(x,y), t>0, x,y\in \Re^d\}$ has a super-convolution property if the sign ``$\leq$'' in \eqref{H0} is changed to ``$\geq$''.
\end{dfn}

\begin{lem}\label{lH10}
Suppose that  the function $\Phi_t(x,y)$  satisfies
\begin{equation}\label{F10}
\big|\Phi_t(x,y)\big| \leq C_{\Phi,T} \Big( t^{-1+\delta_1} H_t^1(x,y)+ t^{-1+\delta_2} H_t^2(x,y)\Big), \quad t\in (0,T], \, x,y\in \rd,
\end{equation}
with some $\delta_1, \,\delta_2\in (0,1)$ and some non-negative kernels $H_t^1(x,y), H_t^2(x,y).$ Assume also that
the kernels $H_t^i(x,y)$, $i=1,2$, satisfy  the sub-convolution property with constant $C_{H,T}$,  and
\begin{equation}\label{H12}
H_t^1(x,y)\geq H_t^2(x,y).
\end{equation}

  Then  for    every      $t\in (0,T]$, $x,y\in \rd$,   the statements below hold true.
\begin{itemize}
\item[a)]   For $k\geq 1$, \normal
\begin{equation}\label{Fk}
\Big| \Phi^{\star k}_t(x,y)\Big| \leq  \frac{C_1C_2^k}{\Gamma(k\zeta)}
t^{-1+(k-1)\zeta} \Big(t^{\delta_1}H_t^{1}(x,y) + t^{\delta_2}  H_t^2(x,y)\Big),
\end{equation}
where
\begin{equation}\label{const}
   C_1 = (3(T\vee 1)C_{H,T})^{-1},  \quad C_2= 3 (T\vee 1)C_{\Phi,T} C_{H,T} \Gamma(\zeta),\quad \text{ and}\quad
\zeta=\delta_1\wedge \delta_2;\normal
\end{equation}

\item[b)] The series $\sum_{k=1}^\infty \Phi_t^{\star k}(x,y)$ is absolutely convergent and
\begin{equation}\label{F20}
\Big|\sum_{k=1}^\infty \Phi_t^{\star k}(x,y)\Big|\leq C  \Big(t^{-1+\delta_1} H_t^1(x,y)+ t^{-1+\delta_2} H_t^2(x,y)\Big);
\end{equation}

\item[c)]
\begin{equation}\label{F25}
\Big|   \Big( H^1 \star \sum_{k=1}^\infty \Phi^{\star k}\Big)_t(x,y)\normal\Big|\leq C  t^{\zeta} H_t^1(x,y).
\end{equation}
\end{itemize}
\end{lem}

\begin{proof}
Using the sub-convolution property of $H_t^{i}(x,y)$ and \eqref{H12}, we get
\begin{equation}\label{H2}
\begin{split}
\big(H^{1}_{t-s}* H^{1}_s\big)(x,y)&\leq  C_{H,T} H_t^{1}(x,y),\\
\big(H^{1}_{t-s}* H^{2}_s\big)(x,y)&\leq  C_{H,T} H_t^{1}(x,y),\\
\big(H^{2}_{t-s}* H^{2}_s\big)(x,y)& \leq  C_{H,T} H_t^{2}(x,y),
\end{split}
\end{equation}
for every $t\leq T, s<t$.

    Observe that  \eqref{Fk} with $k=1$ coincides with  \eqref{F10}\normal. Next, we suppose that \eqref{Fk} holds for $k\geq 1$, and show that it also holds for $k+1$.  Using   \eqref{F10}, \eqref{Fk}, \eqref{H2}, and  the sub-convolution property for $H_t^1(x,y)$, we get
\begin{equation}\label{21}
\begin{split}
|\Phi_t^{\star (k+1)}(x,y)|&=\left|\int_0^t(\Phi^{\star k}_{t-s}*\Phi_s)(x,y)\, ds\right|
\\
&\leq \frac{ C_1C_2^k C_{\Phi,T}}{\Gamma(k\zeta)}\Big\{  \int_0^t (t-s)^{-1+(k-1)\zeta}
(t-s)^{\delta_1}s^{-1+\delta_1}(H_{t-s}^1* H_s^1)(x,y)ds \\
&\quad +\int_0^t  (t-s)^{-1+(k-1)\zeta} (t-s)^{\delta_1} s^{-1+\delta_2} (H_{t-s}^1 *H_s^2)(x,y)ds\\
 &\quad +\int_0^t (t-s)^{-1+(k-1)\zeta} (t-s)^{\delta_2} s^{-1+\delta_1}(H_{t-s}^2*H_s^1)(x,y)ds\\
 &\quad +\int_0^t  (t-s)^{-1+(k-1)\zeta}(t-s)^{\delta_2}s^{-1+\delta_2}(H_{t-s}^2*H_s^2)(x,y)\, ds\Big\}\\
& \leq  \frac{   (T^{\delta_1\vee\delta_2-\zeta}\vee 1)\normal 3C_{\Phi,T} C_1 C_2^k C_{H,T}}{\Gamma(k \zeta)} \Beta(k \zeta , \zeta) t^{-1+k  \zeta\normal}
\Big(  t^{\delta_1} H_t^1(x,y)+ t^{\delta_2} H_t^2(x,y)\Big)
\\
&\leq  \frac{   (1\vee T)\normal 3C_{\Phi,T} C_1 C_2^k C_{H,T}}{\Gamma(k \zeta)} \Beta(k \zeta , \zeta) t^{-1+k  \zeta\normal}
\Big(  t^{\delta_1} H_t^1(x,y)+ t^{\delta_2} H_t^2(x,y)\Big)\\
&=\frac{C_1 C_2^{k+1} }{\Gamma((k+1)\zeta)}  t^{-1+k\zeta}
\Big(  t^{\delta_1} H_t^1(x,y)+ t^{\delta_2} H_t^2(x,y)\Big),
\end{split}
\end{equation}
  which proves \eqref{Fk}\normal.  By  \eqref{Fk},  the series    $\sum_{k=2}^\infty
\Phi_t^{\star (k)} (x,y)$\normal  converge absolutely and
\begin{equation}\label{F30}
\Big|\sum_{k=2}^\infty
\Phi_t^{\star k } (x,y) \Big| \leq C t^{-1+\zeta} \big(t^{\delta_1}H_{t}^{1}(x,y)+t^{\delta_2}H_t^2(x,y)\big),
\end{equation}
which gives \eqref{F20}. Finally, \eqref{F25} follows from  \eqref{F10}, \eqref{F20} and \eqref{H2}.
\end{proof}

In the following proposition we collect the properties of the kernels $H^{(\lambda)}_t(x,y), \lambda\geq 0$ which we require to complete the proof of  Theorem \ref{t1}.  We defer the proof of this proposition to Appendix  B.

\begin{prop}\label{H1} In each of the cases \textbf{A} -- \textbf{C}, for every  $\lambda\in [0,\alpha), T>0$ the kernel   $H_{t}^{(\lambda)}(x,y)$ satisfies
\begin{equation}\label{Hint}
c\leq \int_\rd  H_t^{(\lambda)}(x,y)dy\leq C, \quad t\leq T,
  \end{equation}
  and possesses  the  sub- and super-convolution properties.
\end{prop}
\begin{rem}\label{rem_conv}   Now we can explain why it is convenient to replace in \eqref{Phi1} the kernels $Q^{(\kappa)}_t(x,y)$,
   For $\lambda>0$ the kernel $Q^{(\lambda)}_t(x,y)$  is equal to  zero when $x=\omega(t,y)$, and using this observation we can easily verify  that $Q^{(\lambda)}_t(x,y)$ does not satisfy the sub-convolution property. On the contrary, by Proposition~\ref{H1} the kernels $H^{(\lambda)}_t(x,y), \lambda\geq 0,$  possess the sub-convolution property, hence we can easily derive the bounds for the convolutions powers $\Phi^{\star k}$
 using  Lemma \ref{lH10}. We remark that the sub- and super-convolution properties for  the kernels $H^{(\lambda)}_t(x,y), \lambda\geq 0$ are, in a sense, inherited from the  convolution identity for  the transition probability density of a symmetric $\alpha$-stable process, which is just the  Chapman-Kolmogorov equation for this process:
$$
\int_\rd  g_{t-s}^{(\alpha)}(z-x)g_s^{(\alpha)}(y-z)dz=g_t^{(\alpha)}(y-x), \quad 0<s<t,
$$

An easier, but less precise way to estimate $\Phi_t^{\star k}(x,y)$, $k\geq 1$, dates back to \cite{Ko89}, where  the kernels $Q_t^{(\kappa)}(x,y)$ were  bounded from above by $\frac{C}{t^{d/\alpha}} G^{(\alpha-\kappa)}\left(\frac{y-x}{t^{1/\alpha}}\right)$. These modified kernels possess the sub-convolution property as well. However, their ``tails'' are heavier than those of the $\alpha$-stable density, and therefore such an estimate does not lead to a precise bound for the residue.   The latter weak point was resolved  in \cite{CZ13}, where  the (mixed) $k$-th convolutions of $Q_t^{(\kappa)}(x,y)$ and $Q_t^{(0)}(x,y)$, $k\geq 1$,  were estimated directly, although in a rather cumbersome way.  Our way to estimate the  convolutions is motivated by both approaches,  and inherits their advantages: by using the sub-convolution property we make the overall proof reasonably transparent, and because the ``tails'' of the auxiliary kernels $H_t^{(\kappa)}(x,y)$  and $Q_t^{(\kappa)}(x,y)$ are the same, our upper bounds on $\Phi_t^{\star k}(x,y)$ coincide with those obtained in  \cite{CZ13}.\normal

\end{rem}

\begin{proof}[Proof of statements 1 and 3 of Theorem~\ref{t1}]
We have
\begin{equation}\label{Phi_bound}
|\Phi_t(x,y)|\leq C_{\Phi,T} \Big(t^{-1+\kappa/\alpha}H_t^{(\kappa)}(x,y)+ t^{-1+\delta} H_t^{(0)}(x,y)\Big),
\end{equation}
which is just \eqref{Phi1} modified as we have explained  above. Then we apply Lemma \ref{lH10} with $\delta_1= \kappa/\alpha$, $\delta_2= \delta$,  $H^1_t(x,y)= H_t^{(\kappa)}(x,y)$ and $H_t^2(x,y)= H_t^{(0)}(x,y)$, and get
\begin{equation}\label{Phik}
\Big| \Phi^{\star k}_t(x,y)\Big| \leq   \frac{ C_1C_2^k }{\Gamma(k\zeta)}
t^{-1+(k-1)\zeta} \Big(t^{\kappa/\alpha}H_t^{(\kappa)}(x,y) + t^{\delta}  H_t^{(0)}(x,y)\Big),
\end{equation}
\begin{equation}\label{Psi1}
|\Psi_t(x,y)|\leq C\Big( t^{-1+\kappa/\alpha}H_t^{(\kappa)} (x,y)+ t^{-1+\delta} H_t^{(0)}(x,y)\Big),
\end{equation}
where    $\zeta= \delta\wedge (\kappa/\alpha)$.
In addition,  we have  $p_t^0(x,y)   \asymp  \normal H_t^{(0)}(x,y)$, and therefore by    \eqref{Psi}   and
 \eqref{F25}
\begin{equation}\label{r1}
|r_t(x,y)|\leq C\Big( t^{\kappa/\alpha}H_t^{(\kappa)} (x,y)+ t^{\delta} H_t^{(0)}(x,y)\Big).
\end{equation}
This completes the proof of statement 1.   Since
$$
H_t^{(\kappa)}(x,y)=
\begin{cases}
H_t^{(0)}(x,y),& |\omega(t,y)-x|\leq t^{1/\alpha}\\
\frac{|\omega(t,y)-x|^{\kappa}}{t^{\kappa/\alpha}} H_t^{(0)}(x,y),& t^{1/\alpha}\leq |\omega(t,y)-x|\leq 1,\\
t^{-\kappa/\alpha}H_t^{(0)}(x,y), & |\omega(t,y)-x|\geq 1,
\end{cases}
$$
this also implies \eqref{r_bound} and completes the proof of statement 3.
\end{proof}

The proof of statement 2 is postponed to the next subsection, where the continuity issues are treated in a unified way.

\subsection{Continuity properties of $p_t(x,y)$ and $P_t$}\label{s25}

\begin{proof}[Proof of statement 2 of Theorem~\ref{t1}] In each of the cases \textbf{A} -- \textbf{C}, one can verify directly and easily that the function $p_t^0(x,y)$ and the corresponding $\Phi_t(x,y)$ are continuous in  $(t,x,y)\in (0,\infty)\times\Re^d\times\Re^d$. We show by induction  the continuity of the functions $\Phi^{\star k}_t(x,y)$, $k\geq 2$.

Suppose that $\Phi_t^{\star (k-1)}(x,y)$ is continuous on $(0,\infty)\times\Re^d\times\Re^d$.  Denote
\begin{equation}\label{apr1}
I_R(t,s,x,y):= \int_{B(0,R)} \Phi_{t-s}^{\star (k-1)} (x,z)\Phi_s(z,y)dz,
\end{equation}
where $B(0,R)$ is the ball in $\Re^d$ centered at 0 with radius $R$. By the induction assumption, the function under the integral is continuous in $t,x,y$ for $0<s<t$.

Fix $R_0>0, \tau>0, T>\tau,$ and $\eps\in (0, \tau)$. Then the expression under the integral  is  uniformly continuous in $t,x,y$
for $s\in  [\eps,t-\eps], t\in [\tau,T], x\in  B(0,R_0), y\in B(0,R_0)$, which by the dominated convergence theorem implies the continuity of $I_R(t,s,x,y)$ in the same domain,  if $R>R_0$.

Denote
$$
I(t,s,x,y)=\int_\rd \Phi_{t-s}^{\star (k-1)} (x,z)\Phi_s(z,y)dz,
$$
and observe that
\be\label{conv_33}
|I(t,s,x,y)-I_R(t,s,x,y)|\leq \int_{\rd\backslash B(0,R)}| \Phi_{(t-s)}^{\star (k-1)} (x,z)\Phi_{s}(z,y)|dz\to 0, \quad R\to \infty,
\ee
uniformly in $s\in  [\eps,t-\eps], t\in [\tau,T], x\in  B(0,R_0), y\in B(0,R_0)$. Indeed,  for $R$ large enough and $z\in \rd\backslash B(0,R)$, $y\in B(0,R_0)$, we have $|z-\omega(s,y)|\asymp |z|$, because the function $\omega(s,y)$ is bounded. This implies by \eqref{Phi_bound}
$$
|\Phi_s(z,y)|\leq C(\eps)|z|^{-d-\alpha+\kappa}
$$
for $s\in [\eps,t-\eps]$, $t\in [\tau,T]$, $z\in \rd\backslash B(0,R)$, $ y\in B(0,R_0)$. Since for such $s,\,t,\, x$ and $z$
$$
|\Phi_{t-s}^{\star (k-1)}(x,y)|\leq C(\eps),
$$
convergence \eqref{conv_33} follows by the dominated convergence theorem. This gives that $I(t,s,x,y)$ is continuous in $t,x,y$.

  Since
\be\label{conv_34}
\Phi_t^{\star k}(x,y)=\int_0^t I(t,s,x,y)\, ds=\lim_{\eps\to 0+}\int_\eps^{t-\eps} I(t,s,x,y)\, ds,
\ee
the same argument yields continuity of $\Phi_t^{\star k}(x,y)$. Namely,  proceeding in the same way as in  \eqref{21},  we derive that for $t\in [\tau, T]$
\begin{equation}\label{21a}
\begin{split}
|I(t,s,x,y)|&=\left| \int_\rd \Phi_{(t-s)}^{\star (k-1)} (x,z)\Phi_{s}(z,y)dz\right| \\&\leq  c  (t-s)^{-1+(k-1)\zeta} s^{-1+\zeta} \Big( H_t^{(\kappa)}(x,y)+ H_{t}^{(0)} (x,y)\Big)\\
&\leq C  (t-s)^{-1+(k-1)\zeta} s^{-1+\zeta}, \quad \zeta=\delta\wedge (\kappa/\alpha).
\end{split}
\end{equation}
Hence for every $\eps>0$ the integral in the right hand side of \eqref{conv_34} is continuous by the dominated convergence theorem. Convergence
in \eqref{conv_34} is uniform on compacts in $(0,\infty)\times\Re^d \times \Re^d$; one can easily prove this using \eqref{21a} and the expressions for $H^{(\kappa)}, H^{(0)}$. This completes the proof of continuity of $\Phi_t^{\star k}(x,y)$.

Since the series $\sum_{k=1}^\infty \Phi_t^{\star k}(x,y)$ converges uniformly on compact subsets of $(0,\infty)\times \rd\times \rd$, the function $\Psi_t(x,y)$ is continuous, as well. Continuity of
$$r_t(x,y)=(p^0 \star \Psi)_t(x,y)=p_t(x,y)-p_t^0(x,y)
$$ follows by the same argument as the for $\Phi_t^{\star k}(x,y)$.
\end{proof}

In the rest of the section we derive the basic properties of the family of the operators $P_t$, $t>0$, defined by \eqref{semi}. This is  used in the further analysis of the function $p_t(x,y)$ obtained via the parametrix construction.

Recall that $p_t^0(x,y)=H_t^{(0)}(x,y)$, hence by  \eqref{r1} and \eqref{Hint}, for every $T>0$
\be\label{weight}
\int_\rd p_t(x,y)dy\leq C, \quad t\in (0,T], \, x\in \rd.
\ee
Then the family $P_tf(x), t>0$, is  well defined by \eqref{semi} for  any bounded function $f$.   We also put $P_0f=f$.
  In order to show that   each $P_t$  maps  $C_\infty(\rd)$ into itself, and the family $\{P_t, t\geq 0\}$ is strongly continuous at the point $t=0$,   we need the following proposition.
\begin{prop}\label{H-bound-a}
In each of the cases \textbf{A} -- \textbf{C}, for every $f\in C_\infty(\rd)$
\begin{equation}\label{Fel}
\lim_{|x|\to\infty} \int_\rd p_t^0(x,y)f(y)dy =0, \quad \quad  t>0,
\end{equation}
\begin{equation}\label{Fel2}
\sup_{x\in \rd} \Big|\int_\rd p_t^0(x,y)f(y)dy -f(x)\Big|\to 0, \quad t\to 0.
\end{equation}
\end{prop}
 We defer the proof of this proposition to  Appendix B.
\begin{lem}\label{P-cont} In each of the cases \textbf{A} -- \textbf{C} of Theorem \ref{t1}, the following properties hold true.
\begin{enumerate}\label{cont}
  \item For   every     $t>0$,  $P_t$ is a bounded operator in  $ C_\infty(\rd)$.

  \item For every $f\in C_\infty(\rd)$  we have $\lim_{t\to 0+} \|P_t f-f\|_\infty=0$.

\end{enumerate}
\end{lem}

\begin{proof}[Proof of Lemma~\ref{P-cont}]
1. The proof of continuity of $P_tf$ is    repeats     the proof of continuity of $p_t(x,y)$, and thus is omitted.

   To prove that $P_tf(x)$ vanishes as $|x|\to \infty$, we use  the representation for $p_t(x,y)$, estimate  \eqref{r_bound} on $r_t(x,y)$, and \eqref{Fel}:
 \begin{align*}
\Big| \int_\rd p_t(x,y)f(y)dy \Big| & = \Big|\int_\rd   \Big(p_t^0(x,y)+ r_t(x,y)\Big)f(y)dy\Big| \\& \leq C \int_\rd p_t^0(x,y)|f(y)|dy  \to 0, \quad |x|\to\infty.
\end{align*}
Hence $P_t$ maps $C_\infty(\rd)$ to $C_\infty(\rd)$. Clearly, this operator is linear and bounded (its norm is bounded by the constant $C$ from \eqref{weight}).

2.   By  \eqref{r1} and \eqref{Hint},
\begin{equation}\label{Fel3}
\sup_x \Big|\int_\rd r_t(x,y)f(y)dy\Big|\leq C( t^{\kappa/\alpha}+ t^{\delta})\|f\|_\infty \to0, \quad t\to 0.
\end{equation}
Together with  \eqref{Fel2} this gives the required statement.
\end{proof}

\section{ Proofs of Theorem~\ref{t2} and Theorem~\ref{t3}}  \label{s4}

If we  knew that the function $p_t(x,y)$, constructed in Theorem~\ref{t1}, is a fundamental solution to the Cauchy problem for  $\partial_t-L$, then the properties of $P_t$ stated in Theorem \ref{t2} could be obtained in a standard way based on the positive maximum principle, which holds true for the operator $L$, see \cite[Th.4.5.13]{Ja01}. We refer to   \cite[Ch.4]{EK86} for  the definition and the general results on the positive maximum principle.   \normal   However, on this way we meet substantial difficulties already when we try to prove that $L_x$ can be applied to $p_t(x,y)$. Recall that the domain of $L$ is $C_\infty^2(\rd)$. On the other hand, for the function $p_t^0(x,y)$ we have the following bounds, which can be derived from
 Proposition~\ref{A2}, but yet it seems that they can not be improved:
\begin{equation}\label{ptder}
|\nabla_x p_{t}^0(x,y)|\leq Ct^{-1/\alpha}H_t^{(0)}(x,y),\quad |\nabla^2_{xx} p_{t}^0(x,y)|\leq Ct^{-2/\alpha}H_t^{(0)}(x,y).
\end{equation}
Hence the   spatial    derivatives of $p_t^0(x,y)$ have  non-integrable singularities in $t$ near 0, and such a  behaviour of $p_t^0(x,y)$ makes it unclear why $p_t(x,y)$ should belong to the domain of  $L_x$.

This difficulty is rather typical.    In what follows, we develop an approach which we believe to be well applicable in various situations similar to those  explained above.  The keystone  of this approach is that we use certain \emph{approximate} solution to  the Cauchy problem for  $\partial_t-L$ instead of the exact one.

\subsection{Approximate fundamental solution: construction and basic properties}\label{apr}

For $\eps>0$ we introduce the  auxiliary function
\begin{equation}\label{pe}
p_{t,\epsilon}(x,y):=p_{t+\epsilon}^0(x,y)  + \int_0^t \int_{\rd}  p_{t-s+\eps}^0(x,z)  \Psi_s(z,y) dzds,
\end{equation}
and define
\begin{equation} \label{Pte}
P_{t,\epsilon} f(x):=\int_\rd p_{t,\epsilon}(x,y)f(y)dy, \quad t>0, \, x\in \rd, \quad f\in C_\infty(\rd).
\end{equation}

The  additional time shift by positive $\eps$ removes the singularity at the point $s=t$, and this is the main reason why  $p_{t,\eps}(x,y)$  possesses the following properties:

\begin{itemize}

\item[(i)]   $p_{\cdot,\epsilon}(x,y)\in C^1((0,\infty))$ for any fixed $\epsilon>0$, $x,y\in \rd$;

\item[(ii)]  $p_{t,\epsilon}(\cdot,y)\in C^2_\infty(\rd)$ for any fixed $\epsilon>0$, $t>0$, $y\in \rd$;

\item[(iii)] for any $0<\tau<T$ we have $ p_{t,\epsilon}(x,y)\to p_t(x,y)$  as  $ \epsilon\to 0$, uniformly in
    $(t,x,y)\in [\tau,T]\times \rd\times \rd$.\normal

\item[(iv)] for any $0<\tau<T$ we have
 $$
 q_{t,\epsilon}(x,y):=\big(\prt_t-L_x\big) p_{t,\epsilon}(x,y)\to 0,  \quad \eps\to 0,
 $$
  uniformly in  $(t,x,y)\in [\tau,T]\times \rd\times \rd $.
\end{itemize}

   We do not give the detailed proof of these properties, because everywhere below (with the  only exception of the proof of Proposition~\ref{p1}) we will need the analogues of these properties for $P_{t,\eps}f$, see  Lemma \ref{aux-ep} and Lemma \ref{l5}; the proofs of (i) -- (iv) above are completely analogous and omitted.\normal

Properties (iii),(iv) motivate the name \emph{approximate fundamental solution} we use for $p_{t,\epsilon}(x,y)$: it  approximates $p_{t}(x,y)$ and ``satisfies'' (\ref{L_fund}) in the approximate sense.

\begin{lem}\label{aux-ep}
\begin{enumerate}
\item  For every   $f\in C_\infty(\rd)$, $\eps>0$ the function $P_{t,\eps}f(x)$ belongs to $C^1((0,\infty))$ as a function of $t$ and
to $C^2_\infty(\rd)$ as a function of $x$.
 \item    For every   $f\in C_\infty(\rd)$, $T>0$,
  \be\label{conv_pte}
 \|P_{t,\epsilon}f- P_t f\|_\infty\to 0, \quad \epsilon\to 0,
  \ee
   uniformly in $t\in [0,T ]$, and for every $\eps>0$
  \be\label{conv_x}
  P_{t,\epsilon}f(x)\to 0, \quad |x|\to \infty
  \ee
 uniformly in $t\in [0,T]$.

 \item For $f\in C_\infty(\rd)$ we have
 $$
 \lim_{t,\eps\to 0+} \|P_{t,\eps} f-f\|_\infty =0.
  $$
 \end{enumerate}
\end{lem}
In the proof of this lemma we use the following proposition.
\begin{prop}\label{p0der}
 \begin{enumerate}
   \item    The derivative  $\prt_tp_t^0(x,y)$ exists, and is continuous
 in  $(t,x,y)\in(0,\infty)\times\Re^d\times\Re^d$.\normal

   \item    There exists $C>0$ such that \normal
   $$
   |\prt_tp_t^0(x,y)|\leq C\Big(t^{-1}\vee t^{-1/\alpha}\Big)H_t^{(0)}(x,y),\quad  t>0, \quad x,y\in \Re^d.
   $$
     \end{enumerate}
\end{prop}
   The proof follows directly from the definition of $p_t^0(x,y)$,  the properties of $g^{(\alpha)}(x)$ (cf. Proposition~\ref{A2}),  and the definition of $\omega(t,y)$ in the expression for $p_t^0(x,y)$.\normal

\begin{proof}
Statement 1  can be easily derived by using  the upper bound (\ref{Psi1}) on $\Psi_t(x,y)$, the  upper estimates  on $p_{t}^0(x,y)$,  its space derivatives \eqref{ptder},  time derivatives (see  Proposition~\ref{p0der}), and the dominated convergence theorem.

To prove  statement 2, we first observe that the function
$$
[0,T]\ni t\mapsto \int_{\Re^d}p_{t}^0(\cdot,y)f(y)\, dy\in C_\infty(\Re^d)
$$
is continuous: the continuity at the point $t=0$ is provided by  Proposition \ref{H-bound-a}, and the continuity at all the other points easily  follows by the continuity of $p_{t}^0(x,y)$.  Then
$$
 \int_{\Re^d}p_{t+\epsilon}^0(x,y)f(y)\, dy\to \int_{\Re^d}p_{t}^0(x,y)f(y)\, dy, \quad \epsilon\to 0,
$$
uniformly in $t\in [0,T], x\in \rd$. This together with estimate \eqref{Psi1} and the dominated convergence theorem implies statement 2.

 The proof of statement 3
 is  a slight variation of the proof of statement 2 in Lemma \ref{cont}. Namely, by \eqref{Fel2} we have
$$
\sup_{x\in \rd} \Big|\int_\rd p_{t+\epsilon}^0(x,y)f(y)dy -f(x)\Big|\to 0, \quad t,\eps\to 0.
$$
Hence it is sufficient  to show that
 \begin{equation}\label{Fel33}
\sup_x \Big|\int_0^t \int_\rd \int_\rd \Psi_{t-s+\epsilon }(x,z)p_s^0(z,y)f(y)dydzds\Big| \to0, \quad t,\eps\to 0.
\end{equation}
Using     \eqref{Psi1},   the identity $p^0_t(x,y)=H_t^{(0)}(x,y)$ and the properties $H_t^{(\kappa)}(x,y)$, $H_t^{(0)}(x,y)$, one can verify this relation  similarly to \eqref{Fel3}.
 \end{proof}

Denote
\begin{equation}\label{qte20}
Q_{t,\eps} f(x)= \big(\partial_t-L_x\big) P_{t,\epsilon} f (x), \quad f\in C_\infty(\rd).
\end{equation}

\begin{lem}\label{l5}  For any  $f\in C_\infty(\rd)$  we have
\begin{enumerate}

  \item
  \be\label{conv_loc}
Q_{t,\eps} f(x)\to 0, \quad \epsilon\to 0,
\ee
uniformly in  $(t,x)\in [\tau,T]\times \rd$ for    every    $\tau>0$, $T>\tau$;

\item
 \be\label{conv_int}
\int_0^t Q_{s,\eps} f(x)ds\to 0, \quad \epsilon\to 0,
\ee
uniformly in  $(t,x)\in [0,T]\times \rd$ for any $T>0$.
\end{enumerate}
\end{lem}
\begin{proof} We have
\begin{equation}\label{LP}
LP_{t,\eps}f(x)=L_x \int_\rd  p_{t+\epsilon}^0(x,y)f(y)dy + L_x\int_0^t \int_\rd\int_\rd p_{t-s+\epsilon}^0(x,z)\Psi_s(z,y)f(y)dydzds;
\end{equation}
note that $P_{t,\eps}f$  and both integrals in the right hand side are $C_\infty^2$-functions in $x$ (see Lemma \ref{aux-ep} for the first term; the argument for the second integral is the same). Hence the operator $L_x$ in \eqref{LP} is well applicable. We would like to interchange $L_x$ with the integrals in \eqref{LP}, i.e. to write \begin{equation}\label{LP1}
LP_{t,\eps}f(x)=\int_\rd  L_x p_{t+\epsilon}^0(x,y)f(y)dy + \int_0^t \int_\rd\int_\rd L_x p_{t-s+\epsilon}^0(x,z)\Psi_s(z,y)f(y)dydzds.
\end{equation}
Recall that $L$ is an integro-differential operator given by (\ref{symbol}), and observe that the argument based on the dominated convergence theorem allows us to interchange  the ``differential part'' $(b(x),\nabla)$ of this operator with the integrals in \eqref{LP}. To do the same with the ``integral part'' $a(x)L^{(\alpha)}$, recall that
$$
 L^{(\alpha)}f(x)=\lim_{\eps\to 0+}L^{(\alpha, \eps)}f(x), \quad L^{(\alpha, \eps)}f(x):=\int_{|u|>\eps}\Big(f(x+u)-f(x)\Big)\frac{c_\alpha}{|u|^{d+\alpha}}du,
  $$
and we interchange $a(x)L^{(\alpha, \eps)}$ with the integrals just using the Fubini theorem. On the other hand,
  $$\ba
 | L^{(\alpha)}f(x)-L^{(\alpha, \eps)}f(x)|&=\left|\int_{|u|\leq \eps}\Big(f(x+u)-f(x)-1_{|u|\leq 1}(u, \nabla f)\Big)\frac{c_\alpha}{|u|^{d+\alpha}}du\right|
 \\&\leq C\sup_{x\in \Re^d}|\nabla^2 f(x)|\int_{|u|\leq \eps}|u|^2 \frac{c_\alpha}{|u|^{d+\alpha}}du.
 \ea
  $$
The integrals in \eqref{LP} and the expressions under these integrals belong to $C_\infty^2(\rd)$ in $x$ and their second derivatives admit explicit bounds, cf. \eqref{ptder} and Lemma \ref{aux-ep}. Hence if we put in the right hand side of \eqref{LP} operators
$$
(b(x),\nabla_x)+a(x)L^{(\alpha, \eps)}_x,
$$
instead of $L_x$,  we get the expressions which tend to the original expressions as $\eps\to 0$. The same is true for  \eqref{LP1}, and    since   we  already have  proved that we can interchange $(b(x),\nabla_x)+a(x)L^{(\alpha, \eps)}_x$ with the integrals, we finally obtain \eqref{LP1}.

Similarly, using the differentiability of $p_{t}^0(x,y)$ in $t$ and the upper estimate on the respective derivatives (see Lemma~\ref{p0der} above), we derive
\begin{equation}\label{TP1}
\begin{split}
\prt_t P_{t, \epsilon}f(x)&=\int_\rd \prt_tp_{t+\epsilon}^0(x,y)f(y)dy + \int_0^t \int_\rd\int_\rd  \prt_t p_{t-s+\epsilon}^0(x,z)\Psi_s(z,y)f(y)dydzds\\
&\quad + \int_\rd\int_\rd p_{\epsilon}^0(x,z)\Psi_t(z,y)f(y)dydz.
\end{split}
\end{equation}
Since
$$
(L_x-\prt_t)p_t^0(x,y)=\Phi_t(x,y),
$$
combining \eqref{LP1} and \eqref{TP1} we derive
\begin{equation}\label{qte10}
\begin{split}
Q_{t,\epsilon}f(x)& =\int_\rd \int_\rd p_{\epsilon}^0(x,z)\Psi_t (z,y)f(y)dydz- \int_\rd \Phi_{t+\epsilon} (x,y)f(y)dy\\
&\quad -\int_0^t \int_\rd\int_\rd  \Phi_{t-s+\epsilon}(x,z) \Psi_s(z,y)f(y)dydzds.
\end{split}
\end{equation}
Since the  function $\Psi$ satisfies the equation
$$
\Phi_t(x,y)=\Psi_t(x,y)-\int_0^t\int_{\Re^d}\Phi_{t-s}(x,z)\Psi_s(z,y)\,dzds,
$$
we can rewrite $Q_{t,\epsilon}f(x)$ as follows:
$$\ba
Q_{t,\epsilon}f(x)& = \int_\rd \left(\int_\rd p_{\epsilon}^0(x,z)\Psi_t (z,y)dz-\Psi_{t+\epsilon} (x,y)\right)f(y)dy\\
&\qquad +\int_\rd \left(\int_t^{t+\epsilon} \int_\rd \Phi_{t-s+\epsilon} (x,z)\Psi_s(z,y)dzds\right) f(y)dy\\&=:Q_{t,\epsilon}^1f(x) +Q_{t,\epsilon}^2f(x).\ea
$$
By the uniform  continuity of  $\Psi$ on compact subsets of $(0,\infty)\times \rd\times \rd$ and estimate \eqref{Psi1},  we have
$$
\sup_{t\in [\tau, T], x\in \Re^d}\left|\int_{\Re^d}\Psi_{t+\epsilon} (x,y)f(y)\, dy-\int_{\Re^d}\Psi_{t} (x,y)f(y)\, dy\right|\to 0, \quad \eps\to 0.
$$
Using again the uniform continuity of $\Psi$ on compact subsets of  $(0,\infty)\times \rd\times \rd$, relation  \eqref{Fel2} and  estimate   (\ref{Psi1}), we obtain
$$
\sup_{t\in [\tau, T], x\in \Re^d\times \rd}\left|\int_{\Re^d} p_\eps^0(x,z)\Psi_{t} (z,y)\, dz-\Psi_{t} (x,y)\right|\to 0, \quad \eps\to 0,
$$
$$
\sup_{t\in [\tau, T], x\in \Re^d}\left|\int_{\Re^d}\int_{\Re^d}p_\eps^0(x,z)\Psi_{t} (z,y)f(y)\, dzdy-\int_{\Re^d}\Psi_{t} (x,y)f(y)\, dy\right|\to 0, \quad \eps\to 0.
$$
This proves  (\ref{conv_loc}) with   $Q_{t,\epsilon}^{1}f(x)$ instead of $Q_{t,\epsilon}f(x)$. By (\ref{Psi1})  we have
$$
|Q_{t,\epsilon}^{1}f(x)|\leq Ct^{-1+\zeta}, \quad \zeta=\delta\wedge (\kappa/\alpha),
$$
hence \eqref{conv_int}  for $Q_{t,\eps}^1f(x)$ easily follows from  (\ref{conv_loc}).

By  \eqref{Phi_bound}, \eqref{Psi1} and inequality  $H^{(0)}\leq H^{(\kappa)}$,  we obtain in the same way as in \eqref{21}
\begin{equation}\label{int-sing}
\begin{split}
\int_t^{t+\epsilon} \int_\rd |&\Phi_{t-s+\epsilon} (x,z)\Psi_s(z,y)|dzds   \leq C  H_{t+\epsilon}^{(\kappa)} (x,y)   \int_t^{t+\epsilon}   ((t+\epsilon-s)s)^{-1+\zeta}ds
\\&\leq C  H_{t+\epsilon}^{(\kappa)} (x,y)   t^{-1+\zeta}\int_t^{t+\epsilon}   (t+\epsilon-s)^{-1+\zeta}ds\leq C  \epsilon^{\zeta} t^{-1+\zeta} H_{t+\epsilon}^{(\kappa)}(x,y).
\end{split}
\end{equation}
This immediately  gives  (\ref{conv_loc})  and (\ref{conv_int}) with   $Q_{t,\epsilon}^{2}f(x)$ instead of $Q_{t,\epsilon}f(x)$:  we  multiply (\ref{int-sing}) by $|f(y)|$ and integrate it either with respect to  $dy$ (for \eqref{conv_loc}), or with respect to  $dyds$ (for \eqref{conv_int}).    
\normal
\end{proof}

\subsection{Positive maximum principle, applied to the approximate fundamental solution. Proof of Theorem~\ref{t2}}\label{s52}

The proof of Theorem~\ref{t2} will follow from Lemmae~\ref{posit} -- \ref{Ptint} given below.
\begin{lem}\label{posit} The operator  $P_t$ defined in \eqref{semi} is positivity preserving, i.e. $P_t f \geq 0 $ if $f\geq 0$.
\end{lem}

\begin{proof}
Take $f\in C_\infty(\rd)$, $f\geq 0$, and suppose that
\be\label{ass}
\inf_{t,x} P_t f(x)<0.
\ee
Then there exists $T>0$ such that
$$
\inf_{t\leq T,x\in \Re^d} P_t f(x)<0.
$$
Then by \eqref{conv_pte} there  exist $\eta>0, \theta>0, \eps_1>0$ such that
$$
\inf_{t\leq T,x\in \Re^d} \Big(P_{t,\eps} f(x)+\theta t\Big)<-\eta, \quad \eps<\eps_1.
$$
Denote
$$
u_{\epsilon}(t,x)= P_{t,\epsilon} f(x)+\theta t,
$$
and note that by \eqref{conv_x}
$$
u_{\epsilon}(t,x)\to \theta t>0, \quad |x|\to\infty,
$$
uniformly in $t\in [0, T]$. Hence the above infimum is in fact attained at some point   in $[0,T]\times \rd$\normal; in what follows we fix one such a point for each $\eps$, and denote it by   $(t_\eps, x_\eps)$.

Since $f(x)\geq 0$, by statement 2 of Lemma \ref{aux-ep} there exist $\eps_0>0$, $\tau>0$ such that
$$
P_{t,\eps} f(x)+\theta t\geq -{\eta\over 2}, \quad t\leq \tau, \quad \eps<\eps_0, \quad x\in \Re^d.
$$
   Since
$$
u_{\epsilon}(t_\eps,x_\eps)=\min_{t\in [0, T], x\in \Re^d}u_\eps(t,x)<-\eta<-{\eta\over 2},
$$
we have   $t_\eps>\tau$ as soon as $\eps<\eps_0$.

The operator $L$ satisfies the positive maximum principle; that is, if  whenever $f\in D(L)$, and  $f(x_0)\geq 0$ where $x_0=\arg\max f(x)$, then $Lf(x_0)\leq 0$, cf. \cite[Ch.~4.2]{EK86}. Therefore
$$
L_x u_\eps(t,x)|_{(t,x)=(t_\eps,x_\eps)}\geq 0.
$$
In addition,  for $\eps<\eps_0$ we always have
$$
\prt_t u_\eps(t,x)|_{(t,x)=(t_\eps,x_\eps)}\leq 0,
$$
    where    the  sign ``$<$''  may appear only  if $t_\eps=T$. \normal

Then
\be\label{PMP}
(\prt_t-L_x) u_\eps(t,x)|_{(t,x)=(t_\eps,x_\eps)}\leq 0.
\ee
On the other hand,    since    $t_\eps\in [\tau, T]$, $\eps<\eps_0$, we have  by the first  statement of  Lemma \ref{l5}
$$
(\prt_t-L_x) u_\eps(t,x)|_{(t,x)=(t_\eps,x_\eps)}=\theta +Q_{t_\eps,\eps} f(x_\eps)  \to \theta>0, \quad \eps\to 0.
$$
This gives contradiction and shows that \eqref{ass} fails.
\end{proof}

\begin{lem}\label{semigr}
The family of operators possesses the semigroup property: $P_{t+s} = P_sP_t$.

\end{lem}
\begin{proof} The proof is based on the same argument as the proof of Lemma \ref{posit}, hence we just sketch it. Take $f\in C_\infty(\rd)$ and assume, for instance,  that
\be\label{ass1}
P_{t+s}f(x)-P_tP_sf(x)<0
\ee
for some $s,t>0, x\in \Re^d$. Fix this $s$, and observe that then for some $\eta>0, \theta>0, T>0$ the function
$$
u_{\eps}(t,x)=P_{t+s, \eps}f(x)-P_{t, \eps}P_{s}f(x)+\theta t
$$
satisfies
$$
\inf_{t\leq T,x\in \Re^d}u_{\eps}(t,x)<-\eta.
$$
In addition, $u_\eps(t,x)\to \theta t,|x|\to \infty, $ hence the infimum is attained at some point $(t_\eps, x_\eps)$. Finally,
$u_\eps(t,x)\to 0$, $t,\eps\to 0$ and therefore there exist $\tau>0, \eps_0>0$ such that $t_\eps>\tau$,    provided that $ \eps\in (0,\eps_0)$\normal.
Then, on one hand, \eqref{PMP} holds, and
 on the other hand
$$
(\prt_t-L_x) u_\eps(t,x)|_{(t,x)=(t_\eps,x_\eps)}=Q_{t_\eps+s,\eps} f(x_\eps) - Q_{t_\eps,\eps} P_{s}f(x_\eps)+\theta\to \theta>0, \quad \eps\to 0.
$$
This gives the contradiction and proves that (\ref{ass1}) is impossible. Changing $f$ to $-f$, we see that    inequality \eqref{ass1} with ``$>$'' instead of ``$<$''    is impossible as well, which completes the proof.
\end{proof}

\begin{lem}\label{Ptint}
We have
\begin{itemize}

  \item[a)]
 \begin{equation}\label{Dy1}
P_tf(x)-f(x)=\int_0^tP_sLf(x)\, ds, \quad f\in C_\infty^2(\rd);
\end{equation}

\item[b)]  $$ P_t 1 =1. $$

\end{itemize}

\end{lem}
\begin{proof} We apply the same argument as in the above lemmas. Take $f\in C_\infty^2(\Re^d)$, and assume that    for every $t>0$, $x\in \rd$\normal,
\be\label{ass2}
P_{t}f(x)<f(x)+\int_0^tP_{s}Lf(x)\, ds.
\ee
Then repeating  the above argument, we obtain the functions
 $$
u_{\eps}(t,x)= P_{t, \eps}f(x)-f(x)-\int_0^tP_{s,\eps}Lf(x)\, ds+\theta t, \quad \eps>0
$$
and the points $(t_\eps, x_\eps)$, in which  these functions attain their minima on $[0,T]\times \Re^d$, such that
for some $\tau>0,\eps_0>0$ we have $t_\eps>\tau$,    provided before $ \eps\in (0,\eps_0)$.\normal

On one hand, for these functions we have \eqref{PMP}. On the other hand,
 $$\ba
(\prt_t-L_x) u_\eps(t,x)&= Q_{t, \eps}f(x)+Lf(x)-P_{t,\eps}Lf(x)+\int_0^t LP_{s,\eps}Lf(x)\, ds+\theta
\\& =Q_{t, \eps}f(x) +Lf(x)-P_{t,\eps}Lf(x)+\int_0^t\Big(\prt_sP_{s,\eps}Lf(x)-Q_{s,\eps}Lf(x)\Big)\, ds+\theta
\\&=Q_{t, \eps}f(x) -\int_0^tQ_{s,\eps}Lf(x)\, ds+\theta.
\ea
$$
Then by Lemma \ref{l5}
$$
(\prt_t-L_x) u_\eps(t,x)|_{(t,x)=(t_\eps,x_\eps)}\to \theta>0, \quad \eps\to 0,
$$
which gives contradiction to \eqref{PMP} and disproves (\ref{ass2}). Changing $f$ to $-f$, we complete the proof of statement a).

 To prove statement b),  take $f\in C^2_\infty(\Re^d)$ such that $f\equiv 1$ on the unit ball in $\Re^d$, and put $f_k(x)=f(k^{-1}x)$. Then
$$
f_k(x)\to 1, \quad Lf_k(x)\to 0, \quad k\to \infty,
$$
and both $|f_k|$ and $|Lf_k|$ are bounded by some constant, independent of $k$. Then using the equality $p_t^0(x,y)=H_t^{(0)}(x,y)$, the estimate (\ref{r_bound}) on the remainder $r_r(x,y)$, and Proposition~\ref{H1},  we can apply  the dominated convergence theorem and  pass to  the limit in  \eqref{Dy1} as $k\to \infty$ in  (\ref{Dy1}) with $f_k$ instead $f$. This finishes the proof of statement b).
\end{proof}

\begin{proof}[Proof of  Theorem \ref{t2}.]

By Lemmas~\ref{posit}, \ref{semigr} and  the second statement of Lemma~\ref{Ptint},  the family of operators $\{P_t, t\geq0\}$ forms a strongly continuous contraction semigroup    on $C_\infty(\rd)$\normal, which is positivity preserving.  Since the semigroup $\{P_t, t\geq0\}$ possesses the continuous transition probability density $p_t(x,y)$, the respective Markov process $X$ is strong Feller.  Finally, the first  statement of Lemma~\ref{Ptint} implies that the restriction of the generator of $\{P_t, t\geq0\}$ coincides with $L$ on functions from $C_\infty^2(\rd)$.

\end{proof}

\subsection{The generator of the semigroup $(P_t)_{t\geq0}$: Proofs of Theorem~\ref{t3} and Proposition \ref{p1}}\label{s53}

In Lemma~\ref{Ptint} we proved that
$(L, C_\infty^2 (\rd))$ is the restriction of $(A, D(A))$. Since $A$ is a closed operator, this yields that $(L, C_\infty^2 (\rd))$ is closable. Let us show that its closure coincides with  $(A, D(A))$.

Take $f\in C_\infty(\Re^d)\cap D(A)$.
Fix $t>0$, and consider the functions  $P_tf$ and $P_{t,\eps} f$.  Since $f\in D(A)$, then $P_t  f\in D(A)$, and
 \begin{equation}\label{Pta}
 AP_t f=\partial_t P_tf.
 \end{equation}
  Recall that by statement 1 of Lemma \ref{aux-ep} we have $P_{t,\eps}f \in C_\infty^2(\rd)$ and thus by Lemma~\ref{Ptint} $P_{t,\eps}f \in D(A).$ Hence,
     $$
  AP_{t, \eps} f=L P_{t,\eps}f=\partial_t P_{t,\epsilon}f.
  $$
  Observe that
  \begin{itemize}

  \item    By statement 2 of  Lemma \ref{aux-ep},   one has $P_{t,\eps}f\to P_tf$ in $C_\infty(\Re^d)$ as $\eps\to 0$;

  \item  By  statement 1 of  Lemma \ref{l5},  one has $(\prt_t-L)P_{t,\eps}f\to 0$ in $C_\infty(\Re^d)$ as  $\eps\to 0$.
\end{itemize}

Assuming  that we know
\begin{equation}\label{dPte}
\prt_t P_{t,\eps}\to \prt_tP_t f\quad \text{  in $C_\infty(\Re^d)$ as $\eps\to 0$},
\end{equation}
 we derive
$$
LP_{t,\eps}f\to AP_tf\quad \text{in $C_\infty(\Re^d)$ as $\eps\to 0$},
$$
 which implies that $P_tf $ belongs to the domain of the $C_\infty(\Re^d)$-closure of $(L, {C}_\infty^2 (\rd))$.   Consequently, this closure coincides with  $(A, D(A))$.

We have proved   Theorem \ref{t3} under the assumption \eqref{dPte}. We verify this assumption  in Lemma~\ref{dptx} below.

We also show in Lemma~\ref{dptx}    that\normal
\begin{equation}\label{dpte}
\partial_t p_{t,\eps}(x,y)\to \prt_tp_t(x,y), \quad \eps\to 0,
\end{equation}
uniformly on compact subsets of $(0,\infty)\times \rd\times \rd$.  Using properties i), ii) and iv) of the function $p_{t,\epsilon}(\cdot,y)\in C_\infty^2(\rd)$ (cf. Subsection~\ref{apr}), and applying to this function
literally the same argument used in the case of  $P_{t,\epsilon} f(\cdot)$, we derive that  $p_t(\cdot,y)\in D(A)$, and  $p_t(x,y)$ is a fundamental solution to the Cauchy problem for $\partial_t-A$.
\qed

\section{Proof of Theorem \ref{tweak}}\label{sweak}

   Let $X$ be the canonical Markov process    which corresponds to the semigroup    constructed in Theorem~\ref{t2}\normal.
Using the Markov property of $X$, it is easy to deduce from \eqref{Dy1}  and the  semigroup property for $p_t(x,y)$ the following: For given $f\in C_\infty^2(\Re^d)$, $t_2>t_1$, and $x\in\Re^d$, for any $m\geq 1$, $r_1, \dots r_m\in [0, t_1]$, and bounded measurable $G:(\Re^d)^m\to \Re$ the identity
$$
\E_x\left[f(X_{t_2})-f(X_{t_2})-\int_{t_1}^{t_2}h_f(X_s)\, ds\right]G(X_{r_1}, \dots, X_{r_m})=0
$$
holds true.   Here and below  $\E_x $ denotes the expectation with respect to the law of the underlying process,  starting at $x\in \rd$\normal. This means that for every $f\in C_\infty^2(\Re^d)$ the process
\be\label{Mf}
M^f_t=f(X_t)-\int_{0}^{t}h_f(X_s)\, ds, \quad t\geq 0
\ee
is a $\P_x$-martingale for every $x\in \Re^d$; that is,    $\P$   is a solution to the martingale problem for $(L, C_\infty^2(\rd))$.

 Operator $(L,C_\infty^2(\rd))$ is dissipative, which follows from the positive maximum principle, see \cite[Lem.~4.2.1]{EK86}, or \cite[Lem.~4.5.2]{Ja01}.  Next, its closure equals to the generator $A$ of $C_\infty(\Re^d)$-semigroup $\{P_t, t\geq 0\}$, hence for every $\lambda>0$ the range of the resolvent $(\lambda -L)^{-1}$  in $C_\infty(\rd)$ is dense. Hence the required  uniqueness of the solution to the martingale problem $(L,C_\infty^2(\rd))$ follows by \cite[Thm.~4.4.1]{EK86}.

 It follows  from the It\^o formula, that    every \normal   weak solution  to (\ref{SDE}) is a solution to the martingale problem for $(L,C_\infty^2(\rd))$. Since we have already proved that this martingale problem is well posed, this immediately proves the \emph{uniqueness} of the weak solution to (\ref{SDE}). The proof of the \emph{existence} of a weak solution  can be conducted in a standard way, which we outline below.
 \begin{itemize}
   \item Consider the family of equations
   \begin{equation}\label{Xn10}
   dX_t^{(n)} =  b_n (X_t^{(n)})dt + \sigma_n (X_{t-}^{(n)}) d Z_t^{(\alpha)}.
   \end{equation}
   with smooth coefficients $b_n$ and $\sigma_n$,   approximating the coefficients  $b$ and $\sigma$ of (\ref{SDE}). We can choose the approximations $b_n$ and $\sigma_n$ such that the functions $a_n=(\sigma_n)^\alpha, n\geq 1$ and $b_n, n\geq 1$ the constants satisfy \eqref{a_bdd_Hol} and \eqref{b-bdd-Hol} with the same constants.  Since the coefficients in \eqref{Xn10} are smooth, there exists a (strong) solution $X_t^{(n)}$ to \eqref{Xn10}. This solution is  a strong Markov process, admitting the transition probability density $p_t^{(n)}(x,y)$. Note that under our  assumption on the coefficients  the upper bound for the residue term in Theorem~\ref{t1} can be achieved uniformly in $n$, and consequently
    \begin{equation}\label{ptn}
    p_t^{(n)}(x,y)\leq  \frac{C}{t^{d/\alpha}}  g^{(\alpha)}\left(\frac{\omega_n(t,y)-x}{t^{1/\alpha}}\right), \quad t\in (0,T], \quad x,y\in \rd,
    \end{equation}
    where the constant $C>0$ is independent of $n$, and $\omega_n(t,y)$ is given by \eqref{abc} with $b$ replaced by $b_n$, and the flow $\theta_t$ replaced by the corresponding flow.

   \item We show that
       \begin{equation}\label{EK10}
         \Ee_x\normal \left[ |(X_{t+s}^{(n)}, Z_{t+s}^{(\alpha)})-(X_t^{(n)}, Z_t^{(\alpha)})|^\beta|\mathcal{F}_t\right] \leq \rho(s),
       \end{equation}
       where $\mathcal{F}_t:=\sigma\{ Z_s^{(\alpha)},\, s\leq t\}$,  the non-random function $\rho(s)$ tends to $0$ as $s\to 0$,
       $\beta\in (0,\alpha)$ is some constant.   By Theorem~8.6  and Remark~8.7 from \cite[Ch.~3]{EK86}  we deduce that the  sequence $(X^{(n)}, Z^{(\alpha)})$ is weakly compact in $\DD(\ax, (\Re^{d})^2)$.

       Using the inequality $(u+v)^{\beta}\leq 2^{\beta-1} (u^\beta+ v^\beta)$ for positive $u$, $v$ and $\beta$,  and the  Markov property of $(X_t^{(n)},Z_t^{(\alpha)})$ we derive
       \begin{align*}
       \Ee_x &\Big[ |(X_{t+s}^{(n)}, Z_{t+s}^{(\alpha)})-(X_t^{(n)}, Z_t^{(\alpha)})|^\beta|\mathcal{F}_t\Big]
       \\&= \Ee_{(X_t^{(n)},Z_t^{(\alpha)})}\Big[\left|(X_s^{(n)},Z_s^{(\alpha)})-(X_0^{(n)}, Z_0^{(\alpha)})\right|^\beta\Big]\\
       &=
       \Ee_{(X_t^{(n)},Z_t^{(\alpha)})}\Big[\Big(|X_s^{(n)}-X_0^{(n)}|^2+|Z_s^{(\alpha)}- Z_0^{(\alpha)})|^2\Big)^{\beta/2}\Big]\\
       &\leq 2^{\beta/2-1}\Big( \Ee_{X_t^{(n)}}\Big[|X_s^{(n)}-X_0^{(n)}|^\beta\Big] +  \Ee_{Z_t^{(\alpha)}}\Big[|Z_s^{(\alpha)}-Z_0^{(\alpha)}|^\beta\Big]\Big)\\
              &\leq C\sup_x  \int_\rd |y-x|^\beta \frac{1}{s^{d/\alpha}}  g^{(\alpha)}\left(\frac{\omega_n(s,y)-x}{s^{1/\alpha}}\right)dy + C  \int_\rd  |z|^\beta
        \frac{1}{s^{d/\alpha}}  g^{(\alpha)}\left(\frac{z}{s^{1/\alpha}}\right)dz.
       \end{align*}
     We have
     $$
     \int_\rd  |z|^\beta
        \frac{1}{s^{d/\alpha}}  g^{(\alpha)}\left(\frac{z}{s^{1/\alpha}}\right)dz=cs^{\beta/\alpha}.
     $$
    On the other hand, we can decompose $y-x=(\omega_n(s,y)-x)-(\omega_n(s,y)-y)$, and   use the inequality     $|\omega_n(s,y)-y|\leq Cs$,     since    the sequence $b_n(\cdot), n\geq 1$ is uniformly bounded.  Now simple calculation  finally gives \eqref{EK10} with $\rho(s)=C(s^{\beta/\alpha}+s^\beta)$.

   \item By the weak compactness of the sequence $(X^{(n)}, Z^{(\alpha)})$, there exists a weak limit  $(\tilde X, \tilde Z^{(\alpha)})$ of   $(X^{(n)}, Z^{(\alpha)})$.  By   \cite[Thm.~2.2, Rem.~2.5]{KP91}, this weak limit is a  weak solution to (\ref{SDE}).
 \end{itemize}

\qed

\section{Time derivative  of $p_t(x,y)$. Proof of Theorem~\ref{t5} }\label{s3}

\subsection{Outline}

Our goal in this section is  to  prove the existence of the time derivative $\prt_tp_t(x,y)$, and to give estimates for this derivative. We begin with the outline of our approach, and indicate the  main difficulties.

We would like to extend the properties of $p_t^0(x,y)$ stated in Proposition~\ref{p0der} to similar properties of $p_t(x,y)$. We have the integral representation
$$
p_t(x,y)=p_t^0(x,y)+\int_0^t\int_{\Re^d}p_{t-s}^0(x,z)\Psi_s(z,y)\,dzds,
$$
which is just  another form of \eqref{sol_1}; cf. \eqref{Psi}. However, for the required extension we cannot use this representation, because the upper bound for $|\prt_tp_{t-s}^0(x,z)|$  has a non-integrable singularity $(t-s)^{-1}\vee (t-s)^{-1/\alpha}$ at the point $s=t$. Therefore we rewrite the integral representation for $p_t(x,y)$ in the following way:
 \be\label{46}\ba
p_t(x,y)&
=p_t^0(x,y)+\int_0^{t/2}\int_{\Re^d}p_{t-s}^0(x,z)\Psi_s(z,y)\,dzds+\int_{0}^{t/2}\int_{\Re^d}p_{s}^0(x,z)\Psi_{t-s}(z,y)\,dz\, ds.
\ea
\ee
Using this representation, we avoid the annoying  singularities related to $p_t^0(x,y)$, but instead we have to establish the differential properties of $\Psi$ with respect to  the time variable. For this we proceed in the way similar to that used in Section \ref{s21}: first we establish the required  properties for $\Phi$, then for its convolutions, and finally for $\Psi$. The minor difficulty which arises  is that  in case \textbf{C} the function $b(x)$ is not supposed to be from the class  $C^1_b(\rd)$, and therefore $\Phi_t(x,y)$ is not continuously differentiable in  $t$. This difficulty is of completely technical nature, and is resolved by choosing a suitable formulation for  differentiability property of $\Phi_t(x,y)$ and its convolutions.

\qed

\subsection{Time derivatives of $\Phi$, $\Phi^{\star k}$ and $\Psi$. Proof of the convergence in   \eqref{dPte} and \eqref{dpte}. }\label{timeder}

Consider first the following ``smooth'' case.
\begin{lem}\label{deriv_Phi} Assume that either  one of cases \textbf{A} or \textbf{B} of Theorem \ref{t1} holds true, or  case \textbf{C} of Theorem \ref{t1} holds true with an additional assumption $b\in C^1_b(\Re^d)$. Then the  statements below hold true.

\begin{enumerate}
 \item   Function $\Phi_t(x,y)$ defined by (\ref{Phi})  possesses the derivative
$
\prt_t\Phi_t(x,y),
$ which is  continuous on $(0,\infty)\times\Re^d\times\Re^d$.
   \item For any $\kappa\in (0, \eta\wedge \alpha)$ and $T>0$, the derivative $\prt_t\Phi_t(x,y)$ possesses the bound
   \begin{equation}\label{phider1}
   |\prt_t\Phi_t(x,y)|\leq C\Big(t^{-1}\vee t^{-1/\alpha}\Big)\Big(t^{-1+\kappa/\alpha} H_t^{(\kappa)}+ t^{-1+\delta} H_t^{(0)}(x,y)\Big),\quad  t\in(0, T], \quad x,y\in \Re^d.
   \end{equation}
\end{enumerate}
\end{lem}

\begin{proof}  We give the calculations for the case \textbf{C} only;  the other cases are similar and simpler.  Statement 1 follows directly from the explicit formula  (\ref{Phi_c}). To prove statement 2 we estimate separately the derivatives of $\Phi_t^1(x,y)$, $\Phi_t^2(x,y)$ in (\ref{Phi_c}). We have
$$\ba
|\prt_t\Phi_t^1(x,y)|&\leq C\Big|a(x)-a(y)\Big|\bigg\{{1\over t^{d/\alpha+2}}\Big|(L^{(\alpha)} g^{(\alpha)})\left({\theta_t(y)-x\over t^{1/\alpha}a^{1/\alpha}(y)}\right)\Big|\\&\hspace*{1cm}+{1\over t^{d/\alpha+1+1/\alpha}}\Big|(\nabla L^{(\alpha)} g^{(\alpha)})\left({\theta_t(y)-x\over t^{1/\alpha}a^{1/\alpha}(y)}\right)\Big|\\&\hspace*{1cm}+{1\over t^{d/\alpha+2}}\Big|(\nabla L^{(\alpha)} g^{(\alpha)})\left({\theta_t(y)-x\over t^{1/\alpha}a^{1/\alpha}(y)}\right)\Big|\left|{\theta_t(y)-x\over t^{1/\alpha}}\right|\bigg\},
\ea
$$
because $\prt_t\theta_t(y)=-b(\theta_t(y))$, which is bounded. Applying  (\ref{a_hol_in_c}) with $\kappa\in (0, \eta\wedge \alpha)$, we get by  (\ref{g_a_frac}),  (\ref{g_a_frac_der}), and \eqref{G1}--\eqref{G3}  the estimates
\begin{align*}
|\prt_t\Phi_t^1(x,y)|&\leq C \Big|a(y)-a(x)\Big|
 \Big({1\over t^{d/\alpha+2}} +  {1\over t^{d/\alpha+1+1/\alpha}} \Big)g^{(\alpha)}\left({\theta_t(y)-x\over t^{1/\alpha}a^{1/\alpha}(y)}\right)\\
 & \leq
 C\Big(t^{-2}+t^{-1-1/\alpha} \Big)\Big(t^{\kappa/\alpha} H_t^{(\kappa)}(x,y) +t^{\kappa} H_t^{(0)}(x,y)\Big)\\
&\leq C (t^{-1}\vee t^{-1/\alpha}) \Big( t^{-1+\kappa/\alpha}H_t^{(\kappa)}(x,y)+t^{-1+\kappa} H_t^{(0)}(x,y)\Big).
\end{align*}

 Similarly,
$$\ba
|\prt_t\Phi_t^2(x,y)|&\leq Ct^{-d/\alpha-1/\alpha}\Big|(\nabla g^{(\alpha)})\left({\theta_t(y)-x\over t^{1/\alpha}a^{1/\alpha}(y)}\right)\Big|\\
&+Ct^{-d/\alpha-1/\alpha-1}\Big|b(\theta_t(y))-b(x)\Big|\Big|(\nabla g^{(\alpha)})\left({\theta_t(y)-x\over t^{1/\alpha}a^{1/\alpha}(y)}\right)\Big|\\ &+Ct^{-d/\alpha-1/\alpha-1}\Big|b(\theta_t(y))-b(x)\Big|\Big|(\nabla^2 g^{(\alpha)})\left({\theta_t(y)-x\over t^{1/\alpha}a^{1/\alpha}(y)}\right)\Big|\left|{\theta_t(y)-x\over t^{1/\alpha}}\right|\\
 &+Ct^{-d/\alpha-2/\alpha}\Big|b(\theta_t(y))-b(x)\Big|\Big|(\nabla^2 g^{(\alpha)})\left({\theta_t(y)-x\over t^{1/\alpha}a^{1/\alpha}(y)}\right)\Big|,
\ea
$$
where we have used that $\nabla b$  and $\partial_t \theta_t (y)$ are bounded. Therefore,  using the Lipschitz condition for $b$ and (\ref{g_a_der}), (\ref{g_a_der2}), we can write a shorter (and less precise) estimate
$$
|\prt_t\Phi_t^2(x,y)|\leq C(t^{-1}\vee t^{-1/\alpha}) H_t^{(0)}(x,y),\quad t\in (0,T], \, x,y\in \rd,
$$
which combined with the estimate for $\Phi^1_t(x,y)$ completes the proof.
\end{proof}

\begin{lem} Under the condition of Lemma \ref{deriv_Phi}, the following statements  hold true.

\begin{enumerate}
 \item   The functions $\Phi_t^{\star k}(x,y)$ and $\Psi_t(x,y)$, defined by (\ref{Phi})  have  derivatives
$$
\prt_t\Phi_t^{\star k}(x,y),\quad \prt_t\Psi_t(x,y),\quad t>0, \quad x,y\in \Re^d,
$$
 continuous on  $(0,\infty)\times\Re^d\times\Re^d$.

   \item For any $\kappa\in (0, \eta\wedge \alpha)$ and $T>0$ there exist constants $C_3, C_4$ and $C_5$  such that
   \be\label{42}
   |\prt_t\Phi^{\star k}_t(x,y)|\leq {C_3 C_4^k \over \Gamma(k\zeta)}\Big(t^{-1}\vee t^{-1/\alpha}\Big) t^{-1+(k-1)\zeta}\Big( t^{\kappa/\alpha}H_t^{(\kappa)}(x,y)+ t^{\delta}H_t^{(0)}(x,y)\Big),
   \ee
    \be\label{43}
   |\prt_t\Psi_t(x,y)|\leq C_5\Big(t^{-1}\vee t^{-1/\alpha}\Big) \Big( t^{\kappa/\alpha}H_t^{(\kappa)}(x,y)+ t^{\delta}H_t^{(0)}(x,y)\Big),
   \ee
  for all  $t\in (0,T]$, $x,y\in \rd$.
 \end{enumerate}
\end{lem}
\begin{proof} Since the proof is similar to that  of Lemma \ref{lH10}, we only  sketch the argument. Let $C_1, C_2$ be that same as in \eqref{Phik} (see also \eqref{const}). We show that \eqref{42} holds true with
$$
C_3= \frac{C\Gamma(\zeta)}{C_4}\vee\frac{CC_1}{C_{\Phi,T}}\vee \frac{C_1}{\Gamma(\zeta)},\quad C_4= 9(2\vee 2^{1/\alpha})  (T\vee 1)  C_{\Phi,T}C_{H,T}\Gamma(\zeta),
$$\normal where the constant $C>0$ is  the one from  \eqref{phider1}.

 Split
\be\label{44}
\Phi^{\star (k+1)}_t(x,y)=\int_0^{t/2}\int_{\Re^d}\Phi_{t-s}^{\star k}(x,z)\Phi_s(z,y)\,dz ds+\int_0^{t/2}\int_{\Re^d}\Phi_{s}^{\star k}(x,z)\Phi_{t-s}(z,y)\,dz ds.
\ee
By induction, it can be  shown   that  each $\Phi^{\star k}_t(x,y)$ is continuously differentiable in  $t$, and
\be\label{45}\ba
\prt_t\Phi^{\star (k+1)}_t(x,y)&=\int_0^{t/2}\int_{\Re^d}(\prt_t\Phi^{\star k})_{t-s}(x,z)\Phi_s(z,y)\,dz ds+
\int_0^{t/2}\int_{\Re^d}\Phi_{s}^{\star k}(x,z)(\prt_t\Phi)_{t-s}(z,y)\,dz ds\\&+\int_{\Re^d}\Phi_{t/2}^{\star k}(x,z)\Phi_{t/2}(z,y)\, dz.
\ea
\ee
Observe that
$$
(t-s)^{-1}\vee (t-s)^{-1/\alpha}\leq  \big(2\vee 2^{1/\alpha}\big)\Big(t^{-1}\vee t^{-1/\alpha}\Big), \quad s\in (0, t/2).
$$
Using this inequality and   pulling out from the integrals the term  $\big(2\vee 2^{1/\alpha}\big)\big(t^{-1}\vee t^{-1/\alpha}\big)$, we get by the induction assumption
\begin{align*}
\Big|&\int_0^{t/2}\int_{\Re^d}(\prt_t\Phi^{\star k})_{t-s}(x,z)\Phi_s(z,y)\,dz ds\Big|\leq
\frac{C_3 C_4^k C_{\Phi,T} ( 2\vee 2^{ 1/\alpha})}{\Gamma(\zeta k)}  (t^{-1}\vee t^{-1/\alpha}\big) \\
&\quad \times \int_0^t \int_\rd (t-s)^{-1+(k-1)\zeta}\Big( (t-s)^{\kappa/\alpha}H_{t-s}^{(\kappa)}(x,z)+ (t-s)^\delta H_{t-s}^{(0)}(x,z)\Big)\\
&\quad \times \Big( s^{\kappa/\alpha}H_s^{(\kappa)}(z,y)+ s^\delta H_s^{(0)}(z,y)\Big)dzds\\
&\leq \frac{3    (T\vee 1)\normal C_3 C_4^k C_{\Phi,T} C_{H,T}( 2\vee 2^{ 1/\alpha})}{\Gamma(\zeta k)}  (t^{-1}\vee t^{-1/\alpha}\big) \Beta(\zeta k,\zeta)t^{-1+k\zeta} \Big( t^{\kappa/\alpha}H_t^{(\kappa)}(x,y)+ t^\delta H_t^{(0)}(x,y)\Big)\\
&= \frac{3   (T\vee 1) \normal C_3 C_4^k C_{\Phi,T} C_{H,T}( 2\vee 2^{ 1/\alpha})\Gamma(\zeta)}{\Gamma(\zeta (k+1))}  (t^{-1}\vee t^{-1/\alpha}\big)t^{-1+k\zeta} \Big( t^{\kappa/\alpha}H_t^{(\kappa)}(x,y)+ t^\delta H_t^{(0)}(x,y)\Big)\\
&\leq \frac{C_3 C_4^{k+1}}{3\Gamma(\zeta (k+1))}(t^{-1}\vee t^{-1/\alpha}\big)t^{-1+k\zeta} \Big( t^{\kappa/\alpha}H_t^{(\kappa)}(x,y)+ t^\delta H_t^{(0)}(x,y)\Big).
\end{align*}
In the same fashion, it can be shown that
\begin{align*}
\Big|\int_0^{t/2}&\int_{\Re^d}\Phi_{s}^{\star k}(x,z)(\prt_t\Phi)_{t-s}(z,y)\,dz ds
\Big|\\
&\leq \frac{C_3 C_4^{k+1}}{3\Gamma((k+1)\zeta)} (t^{-1}\vee t^{-1/\alpha}\big)t^{-1+k\zeta} \Big( t^{\kappa/\alpha}H_t^{(\kappa)}(x,y)+ t^\delta H_t^{(0)}(x,y)\Big).
\end{align*}
For the  third term we have by \eqref{Phik}
\begin{align*}
\Big| \int_{\Re^d}\Phi_{t/2}^{\star k}(x,z)&\Phi_{t/2}(z,y)\, dz\Big|
\leq \frac{ C_1C_2^k C_{\Phi,T} C_{H,T}}{\Gamma(k\zeta)} \Big(\frac{t}{2}\Big)^{-1+(k-1)\zeta} \\
&\times  \Big\{
\Big(2\Big(\frac{t}{2}\Big)^{\kappa/\alpha+\delta} + \Big(\frac{t}{2}\Big)^{2\kappa/\alpha} \Big) H_t^{(\kappa)}(x,y)+
\Big(\frac{t}{2}\Big)^{2\delta} H_t^{(0)}(x,y)\Big\}\\
&\leq \frac{3    ( T\vee 1)\normal C_1C_{\Phi,T} C_{H,T} C_2^k}{\Gamma(k\zeta)}  t^{-1+k\zeta}
  \frac{2^{1-\zeta}}{(2\epsilon)^{k\zeta}} \normal \Big\{ t^{\kappa/\alpha}H_t^{(\kappa)}(x,y)+ t^{\delta} H_t^{(0)}(x,y)\Big\}.
\end{align*}
From the inequality $u^\zeta\leq e^{ (1-\epsilon)u}$, $u>0$, $\eps\in (0,1)$, we get  the estimate
 $$
\Gamma((k+1)\zeta) = \int_0^\infty e^{-u}u^{(k+1)\zeta-1} du\leq \epsilon^{-k\zeta} \Gamma(\zeta k).
 $$
Without loss of generality we assume that   $\epsilon>1/2$; then  $(2\epsilon)^{-\zeta k}\leq 1$.
 Therefore, we  arrive at
 $$
 \Big| \int_{\Re^d}\Phi_{t/2}^{\star k}(x,z)\Phi_{t/2}(z,y)\, dz\Big|\leq
 \frac{ C_1 C_4^{k+1} }{3\Gamma(\zeta)\Gamma((k+1)\zeta)}  t^{-1+k\zeta}
\Big\{ t^{\kappa/\alpha}H_t^{(\kappa)}(x,y)+ t^{\delta} H_t^{(0)}(x,y)\Big\}.
$$

Adding the obtained  estimate,  we get \eqref{42} with $k+1$ instead of $k$.

 In addition, for fixed $y\in \Re^d$ each term in the sum has a derivative in  $t$, continuous in  $(t,x)\in (0,T]\times \Re^d$,  and by  (\ref{42}) the series for the derivative is also uniformly convergent.   Thus,  $\Psi$ has a  derivative in  $t$, which is  continuous with respect to   $(t,x)\in (0, \infty)\times \Re^d$ and  satisfies (\ref{43}).
\end{proof}

In the above proof, in the case \textbf{C}  we differentiate in $t$ the term
$$
b(\theta_t(y))
$$
in the expression for $\Phi_t^2(x,y)$. If $b$ does not belong to $C^1$, this term may not be continuously differentiable.   Nevertheless, it is possible to show that the above result extends in a certain sense to the case when  $b$ is   only \normal \, assumed to be Lipschitz continuous.

\begin{lem}\label{deriv_irr} In case \textbf{C} of Theorem \ref{t1}, the  following statements hold true.

\begin{enumerate}
 \item   There exists a set $\Upsilon\subset (0, \infty)\times \Re^d$ of zero Lebesgue measure such that the functions  $\Phi_t^{\star k}(x,y)$, $k\geq 1$, and $\Psi_t(x,y)$ are differentiable in $t$ for    every      $x\in \Re^d$ and $(t,y)\not \in \Upsilon$.
   \item For    every      $(t,y)\not\in \Upsilon$, the  time derivatives $\prt_t\Phi_t^{\star k}(x,y), k\geq 1$, and $\prt_t\Psi_t(x,y)$ are continuous in  $x\in \rd$ and satisfy  (\ref{42}), (\ref{43}).
\end{enumerate}
\end{lem}
\begin{proof} Recall that by the  Rademacher theorem (cf. \cite[Thm.~VII.23.2]{B02}) the Lipschitz continuous function $b$  has a   gradient    a.e. with respect to   the Lebesgue measure on $\Re^d$. Denote by $\Upsilon_b$ the exceptional set of zero Lebesgue measure, such that $b$ is differentiable at every point outside  $\Upsilon_b$. Since $\theta_t$ is a diffeomorphism of $\Re^d$ (see Theorem~I.2.3 and  the comment in Chapter I $\S$5 from \cite{CL55}), the set $\Upsilon_{t,b}=\{y:\theta_t(y)\in \Upsilon_b\}$ is again  of zero Lebesgue measure. Since  $\prt_t\theta_t(y)=-b(\theta_t(y))$, the derivative $\prt_tb(\theta_t(y))$ is well defined for    every     $y\in \Upsilon_{t,b}$. This derivative is given by
$$
\prt_tb(\theta_t(y))=-\sum_{j=1}^d\prt_{j}b(\theta_t(y))b_j(\theta_t(y)),
$$
where  the partial derivatives $\prt_jb$ are now well defined on $\Upsilon_b$ and bounded, because $b$ is Lipschitz continuous. The term $b(\theta_t(y))$ comes in the expression for $\Phi$ in a multiplicative way, and  all other terms have  derivatives in $t$, and are  continuous in  $(t,x,y)$.  Hence, repeating the calculations from the proof of Lemma \ref{deriv_Phi}, we   get the (part of) required statements for $\Phi$,  with the exceptional set
$$
\Upsilon^1=\{(t,y):y\in \Upsilon_{t,b}\}.
$$
Further, it is easy to get   by induction the same statements for $\Phi^{\star k}, k\geq 2$, with the exceptional set
$$
\Upsilon=\Upsilon^1\bigcup\left\{(0,\infty)\times \{y:\int_0^\infty1_{y\in \Upsilon_{s,b}}\,ds>0\}\right\}.
$$
Indeed,  by \eqref{44}
\begin{align*}
{\Phi^{\star (k+1)}_{t+\triangle t}(x,y)-\Phi^{\star (k+1)}_t(x,y)\over \triangle t}&=
\int_0^{t/2} \int_\rd \frac{\Phi_{t+\triangle t-s}^{\star k}(x,z)-\Phi_{t-s}^{\star k}(x,z)}{\triangle t} \Phi_s(z,y)dzds\\
&+ \frac{1}{\triangle t} \int_{t/2}^{(t+\triangle t)/2} \int_\rd \Phi_{t+\triangle t-s}^{\star k}(x,z) \Phi_s(z,y)dzds\\
&+\int_0^{t/2} \int_\rd \Phi_s^{\star k} (x,z)\frac{\Phi_{t+\triangle t-s}(z,y)-\Phi_{t-s}(z,y)}{\triangle t} dzds\\
&+ \frac{1}{\triangle t} \int_{t/2}^{(t+\triangle t)/2} \int_\rd  \Phi_s^{\star k}(x,z) \Phi_{t+\triangle t-s}(z,y)dzds.
\end{align*}
Observe that if $(t,y)\notin \Upsilon$ then the respective ratios under the first and the third integrals converge $ds$-a.e. to the derivatives
$(\prt_t\Phi^{\star k})_{t-s}(x,z)$ and $\prt_t\Phi_{t-s}(z,y)$, and the functions $\Phi_{t+\triangle t-s}^{\star k}(x,z)$ and  $\Phi_{t+\triangle t-s}(z,y) $ converge, respectively, to $\Phi_{t-s}^{\star k}(x,z)$ and  $\Phi_{t-s}(z,y) $. Then the convergence of the integrals follows by  dominated convergence theorem and  estimates  (\ref{42}) and (\ref{Phik}). Hence, the derivative $\prt_t\Phi^{\star (k+1)}_t(x,y)$ exists and admits representation (\ref{45}). The bound (\ref{42}) for it follows by induction. Its continuity in $x$ also follows by induction and the dominated convergence theorem.

Similarly, one can obtain the required statement for $\Psi$. Recall that $\Psi_t(x,y)$ is  given by   the (uniformly convergent) series, and for each term both its differentiability in $t$ and the bound (\ref{42})  are proved for $(t, y)\not \in \Upsilon$.  Then by the dominated convergence theorem we get the same properties for the whole sum. To get the continuity with respect to  $x$,  we again use  the dominated convergence theorem.
\end{proof}

 The estimates on the derivatives we just obtained allow us to verify easily assertions \eqref{dPte}, \eqref{dpte}, which play the crucial role in the proof of Theorem~\ref{t3} and Proposition~\ref{p1}.

\begin{lem}\label{dptx}

\begin{enumerate}

  \item  For any  $f\in C_\infty(\rd)$,
$$
\|\prt_tP_{t,\eps} f-\prt_t P_t f\|_\infty \to0,  \quad \epsilon\to 0,
 $$
uniformly on compact subsets of $(0,\infty)$. Moreover, $\prt_tP_tf(x)=\int_\rd \prt_t p_{t}(x,y)f(y)dy$.

 \item  Under the assumptions of Proposition~\ref{p1},
 $$
 \partial_t p_{t,\epsilon}(x,y)\to \partial_t p_t(x,y)\quad \text{as}\quad \epsilon\to 0,
 $$
   uniformly on compact subsets of $(0,\infty)\times \rd\times \rd$;

\end{enumerate}

\end{lem}

\begin{proof}
The proofs of both statements  rely on  decomposition (\ref{46}).
 We prove the first statement; the proof of the second statement is completely similar.

 Using (\ref{46}) we have
\be\label{47}\ba
\prt_t\int_\rd  p_{t,\eps}(x,y)f(y)dy&=\int_\rd  \prt_tp_{t,\eps}(x,y)f(y)dy=\int_\rd \prt_t p_{t+\eps}^0(x,y)f(y)\, dy\\&+\int_0^{t/2}\int_{\Re^d}\int_{\Re^d}(\prt_t p^0)_{t-s+\eps}(x,z)\Psi_s(z,y)f(y)\,dzdyds
\\&+\int_0^{t/2}\int_{\Re^d}\int_{\Re^d} p^0_{s+\eps}(x,z)(\prt_t\Psi)_{t-s}(z,y)f(y)\,dzdyds\\&+\int_{\Re^d}\int_{\Re^d} p^0_{t/2+\eps}(x,z)\Psi_{t/2}(z,y)f(y)\,dzdy.
\ea
\ee
Note that for every positive $t_0<t_1$
$$
p^0_{t+\eps}(x,y)\to p^0_{t}(x,y),\quad \prt_tp^0_{t+\eps}(x,y)\to \prt_tp^0_{t}(x,y), \quad \eps\to 0,
$$
uniformly in $t\in [t_0, t_1], x,y\in \rd$. Now the required convergence follows from \eqref{47} and the bounds for $p^0, \Psi, \prt_t p^0, \prt_t\Psi$ obtained above.
\end{proof}

\subsection{Completion of the proof of Theorem \ref{t5}}

Now we can finalize the proof of Theorem~\ref{t5}. Again, we consider only   the     most cumbersome Case \textbf{C} with $b$ being  Lipschitz continuous.

By  representation (\ref{46}), the first two statements of the theorem follow from the  statements given above on  time derivatives of $p_t^0(x,y)$ and $\Psi_t(x,y)$; the proofs are completely analogous to those of Lemma \ref{deriv_irr}, and therefore are omitted. To prove statement 3, note that the set $\Upsilon$ constructed in Lemma \ref{deriv_irr} is such that for every fixed $t>0$ the set $\{y: (t,y)\in \Upsilon\}$ has zero Lebesgue measure. Together with the bounds for $\prt_tp_t(x,y)$ from statement 2, this makes it possible to use the  dominated convergence theorem  and prove that for  given $t>0$ and $f\in C_\infty(\Re^d)$
$$
{P_{t+\triangle t} f(x)-P_tf(x)\over \triangle t}\to \int_{\Re^d}\prt_tp_t(x,y)f(y)\, dy, \quad \triangle t\to 0,
$$
uniformly in $x\in \Re^d$, which gives statement 3.\qed

\begin{rem}\label{r_Kol} In the above proof of Theorem~\ref{t5}, which is based on (\ref{46}) and the subsequent parametrix-type iteration of convolutions, we are strongly motivated by the idea used in the proof of Theorem 3.1 in \cite{Ko00}. According to this idea,  we decompose the $\star$-convolution  in two parts in such a way, that after such a decomposition the time derivative is applied  to the ``least singular'' function under the integral, as it was done in  \eqref{46} -- \eqref{45}.  Unfortunately, we cannot proceed in the same way with the derivative $\prt_x$ unless  $p_t^0(x,y)$ depends on $t$ and $x-y$ only.  It seems that in this place in \cite{Ko00} there is a mistake hidden in the calculations, because in this part of the proof   the respective (space) convolutions are treated as if they only depend on the difference  of space  arguments, but in fact their structure is more complicated.   Therefore we do not use the above argument from \cite{Ko00} for the derivative $\prt_x$, and develop another way to justify the whole method.

\end{rem}

\section{Proof of Theorem~\ref{t4}}\label{s7}

 Since the function $V_t(x)$ (cf. \eqref{r_bound2}) is bounded, the upper bound in \eqref{uplo} follows just by the definition of $p_t^{(0)}(x,y)$ and \eqref{sol}.

 Let us prove the lower bound. First we observe that if we manage to prove the lower bound  for \emph{some} $T>0$, then we actually can do that for \emph{all} $T>0$. This follows directly from   the super-convolution property of the kernel $H^{(0)}_t(x,y)$ at the right hand side of \eqref{uplo} and the Chapmen-Kolmogorov identity (that is, the convolution identity) for $p_t(x,y)$ at the left hand side.

Note that for $|\omega(t,y)-x|\leq t^{1/\alpha}$  we have by \eqref{r1}
\begin{equation}\label{r5}
|r_t(x,y)|\leq C \big( t^{\kappa/\alpha} + t^\delta \big) H^{(0)}_t(x,y).
\end{equation}
Therefore,  by \eqref{sol} and  \eqref{r5} we get
\begin{equation}\label{low5}
p_t(x,y)\asymp t^{-d/\alpha}, \quad t\in (0,T], \quad |\omega(t,y)-x|\leq t^{1/\alpha}.
\end{equation}
Further, by \eqref{r_bound} and  \eqref{r_bound2} there exists $\rho\in (0,1)$ small, such that  if  $t^{1/\alpha}\leq  |\omega(t,y)-x|\leq \rho$ then
\begin{equation}\label{r10}
|r_t(x,y)|\leq 2^{-1} p_t^{(0)}(x,y).
\end{equation}
This implies
\begin{equation}\label{low10}
p_t(x,y)\geq 2^{-1} p_t^{(0)}(x,y), \quad t\in (0,T], \quad t^{1/\alpha}\leq |\omega(t,y)-x|\leq \rho.
\end{equation}
Let us show that there exists $c>0$ such that
\begin{equation}\label{low20}
p_t(x,y)\geq \frac{ct}{|\omega(t,y)-x|^{d+\alpha}}, \quad t\in (0,T], \quad |\omega(t,y)-x|> \rho.
\end{equation}
 Consider the set
$$
D=\Big\{(s,z): |\omega(t-s,y)-z|<t^{1/\alpha}\Big\}\subset [0,\infty)\times\rd,
$$
and denote
$$
\tau=\inf\{s: (s,X_s)\in D\}.
$$
If $\tau\leq t/2$, then we have $|\omega(t-\tau,y)-X_\tau|\leq  t^{1/\alpha}$, $t-\tau>t/2$, hence  by the strong Markov property and \eqref{low5} we have
$$
p_t(x,y)\geq \Ee_x\left[p_{t-\tau}(X_\tau, y) 1_{\tau\leq t/2}\right]\geq ct^{-d/\alpha}\P_x(\tau\leq t/2).
$$
To estimate $\P_x(\tau\leq t/2)$, we introduce another stopping time $\sigma$ in the following way. Up to now, $T>0$ was fixed but arbitrary. Now we take another $T_1>0$ small enough, so that
$$
t^{1/\alpha}\leq{\rho\over 3}, \quad |\omega(t-s,y)-\omega(t,y)|<{\rho\over 3}, \quad 0\leq s\leq t\leq T_1, \quad y\in \Re^d.
$$
Here in the second inequality we have used that
\be\label{om_der}
\prt_t\omega(t,y)=\begin{cases}
        0,&  \text{in case \textbf{A}},\\
        -b(y), & \text{in case \textbf{B}},\\
        - b(\theta_{t}(y)), & \text{in case \textbf{C}},
        \end{cases}
\ee
and thus $\prt_t\omega(t,y)$ is bounded. Define
$$
\sigma=\inf\left\{s:|X_s-x|\geq {\rho\over 3}\right\}\wedge \left({t\over 2}\right),
$$
then $\sigma\leq t/2$, and if $t\leq T_1$ for every $s<\sigma$ we have
\be\label{calc}
|\omega(t,y)-x|\leq|\omega(t-s,y)-X_s|+ |X_s-x|+|\omega(t,y)-  \omega(t-s,y)\normal|< |\omega(  t-s\normal,y)-X_s|+{2\rho\over 3}.
\ee
   Since    $|\omega(t,y)-x|>\rho,$ we have $|\omega(t  -s\normal,y)-X_s|>\rho/3>t^{1/\alpha}$, i.e. $(s, X_s)\not \in D$. Hence
\be\label{incl}
\{\tau\leq t/2\}\supset\{(\sigma,X_\sigma)\in D\}.
\ee
  Take $f\in C_\infty^2(\rd)$ such that $f(z)\in [0,1]$,
  $$
  f(z)=\left\{
        \begin{array}{ll}
          1,  & |z|\leq 2^{-1}t^{-1/\alpha} \\
          0, & |z|>t^{-1/\alpha},
        \end{array}
      \right.
  $$
  and for a fixed $t\leq T_1, y\in \Re^d$ consider the function $F(s,z)=f(\omega(t,y)-z)$. Then by the It\^o formula and Doob's optional sampling theorem applied to the bounded stopping time $\sigma$, we have
  $$
  \E_x F(\sigma, X_\sigma)=F(0,x)+\E_x\int_0^\sigma \Big(L_xF(s, X_s)+F'_s(s, X_s)\Big)\, ds,
  $$
  where
        \begin{equation*}
        F'_s(s,X_s)=
                \Big(\nabla f(\omega(t-s,y)),\prt_s \omega(t-s,y)\Big),
        \end{equation*}
  see \eqref{om_der} for the formula for $\prt_t \omega(t,y)$. Now we recall that
  \begin{itemize}
    \item[(i)] $F\leq 1,$ and $F(\sigma, X_\sigma)=0$ if $(\sigma, X_\sigma)\not \in D$;
    \item[(ii)] for every $s<\sigma$, $|X_\sigma-x|<\rho/3,$ and therefore by the calculation \eqref{calc} we have $|\omega(t-s,y)-X_s|>t^{1/\alpha}$, which yields that
        $$
        F(s,X_s)=0, \quad \nabla_x F(s,X_s)=0,\quad F'_s(s,X_s)=0.
        $$
  \end{itemize}
Hence by \eqref{incl} we have
$$
\ba
\P_x(\tau\leq t/2)&\geq \E_x\int_0^\sigma a(X_s)L_x^{(\alpha)}F(s,X_s)\, ds
\\&= \Ee_x \left[ \int_0^\sigma a(X_s) \int_{|\omega(t-s,y)-(X_s+u)|\leq t^{1/\alpha}} \frac{c_\alpha }{|u|^{d+\alpha}}\,duds\right]
\\&\geq c  \Ee_x\int_0^\sigma \frac{t^{d/\alpha}}{|\omega(t-s,y)-X_s|^{d+\alpha}}\, ds.
\ea
$$
Observe that by \eqref{calc} $|\omega(t-s,y)-X_s|\geq 3^{-1}|\omega(t,y)-x|,$ hence
$$
\P_x(\tau\leq t/2)\geq \frac{Ct^{d/\alpha+1}}{|\omega(t,y)-x|^{d+\alpha}}\P_x\left(\sigma>{t\over 4}\right).
$$
It is easy to verify that by choosing $T_1$ small enough we can ensure that $\P_x\left(\sigma>{t/4}\right)>1/2$ for $t<T_1, x\in \Re^d$. Summarizing all the calculations above we get the required bound \eqref{low20}.

\qed

\section*{Appendix A: Proof of Proposition~\ref{A2}}\label{sB}

 Estimate \eqref{g_a} for the  $\alpha$-stable transition probability density is well known, see, for example,  \cite{PT69}, \cite{St10a}--\cite{St11}, \cite{W07}; see also  \cite{Zo86} for the asymptotic behaviour of an $\alpha$-stable distribution density in the one-dimensional case.

   Inequality   \eqref{g_a_der}  was proved in \cite[Lem.5]{BJ07}.  The proof therein is based on the  subordination argument, i.e.  on the representation of $Z^{(\alpha)}$ as a Brownian motion with a time change performed by an independent one-sided $\alpha/2$--stable process. The same approach can be applied to the proof of \eqref{g_a_der2}; since the proof  follows  literally  the proof of  \eqref{g_a_der} in  \cite[Lem.5]{BJ07}, we omit the details.

Let us show \eqref{g_a_frac} and \eqref{g_a_frac_der}.
Recall that   $g_t^{(\alpha)}(y-x)=\frac{1}{t^{d/\alpha}} g^{(\alpha)}\big(\frac{y-x}{t^{1/\alpha}}\big)$ is the  transition probability density of $Z^{(\alpha)}$, and therefore
\begin{equation}\label{La}
L^{(\alpha)} g_t^{(\alpha)} (x)= \prt_t g_t^{(\alpha)}(x)= -\frac{d}{\alpha t^{d/\alpha+1}} g^{(\alpha)}\Big( \frac{x}{t^{1/\alpha}}\Big) - \frac{1}{\alpha t^{(d+1)/\alpha+1}}\left(x,  \nabla g^{(\alpha)}\Big( \frac{x}{t^{1/\alpha}}\Big)\right).
\end{equation}
Now  \eqref{g_a_frac} follows from \eqref{g_a}, \eqref{g_a_der} and  \eqref{La} with $t=1$.  Differentiating  \eqref{La} in $x$, taking $t=1$,  and  applying \eqref{g_a_der} and \eqref{g_a_der2}, we get \eqref{g_a_frac_der}. \normal
\qed

\section*{Appendix B: Proof of Propositions~\ref{H1} and   \ref{H-bound-a} }

\begin{proof}[Proof of Proposition~\ref{H1}] We prove the sub-convolution property, only: the proof of the super-convolution   property    is completely analogous and is omitted.

  Recall that $H_t^{(\lambda)}(x,y)$ is  defined in \eqref{Ht21}, where $\omega(t,y)$ is given in \eqref{abc} for each of the cases  \textbf{A} -- \textbf{C}. In what follows, we fix $\lambda\in [0,\alpha)$, and omit it in the notation, i.e. write $H_t(x,y)$ instead of $H^{(\lambda)}_t(x,y)$. We keep the same notation   $H_t(x,y)$ for each of the cases \textbf{A} -- \textbf{C}, but have in mind, that it is defined according to \eqref{abc}.

     Define
\begin{equation}\label{Kt}
K_t(x):=\left(\left( \big|\frac{x}{t^{1/\alpha}}\big|^\lambda \vee 1 \right) \wedge t^{-\lambda/\alpha}\right) \frac{1}{t^{d/\alpha}}  G^{(\alpha)} \left({x\over t^{1/\alpha}}\right).
\end{equation}
\normal

\emph{Case \textbf{A}. }  Note that in case \textbf{A} the  kernel $H_t(x,y)$ depends on the difference $y-x$, only,  which immediately gives \eqref{Hint}.
 Let us show the sub-convolution property.

Note that
\begin{equation}\label{Hb10}
K_t(x)\leq \frac{1}{t^{d/\alpha}} G^{(\alpha-\lambda)}\left(\frac{x}{t^{1/\alpha}}\right),
\end{equation}
and
\begin{equation}\label{Hb20}
K_t(x)= \frac{1}{t^{d/\alpha}} G^{(\alpha-\lambda)}\left(\frac{x}{t^{1/\alpha}}\right) \quad \text{if} \quad |x|\leq 1.
\end{equation}
On the other hand, by \eqref{g_a} we have      $\frac{1}{t^{d/\alpha}} G^{(\alpha-\lambda)}\left(\frac{x}{t^{1/\alpha}}\right)\asymp g^{(\alpha-\lambda)}_{t^{1-\lambda/\alpha}}(x)$\normal. Since the function $g_t^{(\alpha -\lambda\normal)}(y-x)$ is the transition probability density of an $(\alpha-\lambda)$ --stable process $Z^{(\alpha-\lambda)}$,  it possesses the convolution property;   see Remark~\ref{rem_conv}\normal.
 Therefore, if $|x|\leq 1$,  we have
 $$
 (K_{t-s}* K_s)(x)\leq C K_{(t-s)^{1-\lambda/\alpha} +s^{1-\lambda/\alpha}}(x)
 $$
 with the constant $C>0$ depending  on $\alpha, \lambda$, and $d$ only. Observe that
 $$
t^{1-\lambda/\alpha}\leq (t-s)^{1-\lambda/\alpha} +s^{1-\lambda/\alpha}\leq 2t^{1-\lambda/\alpha}, \quad 0\leq s\leq t.
 $$
 Thus, it follows from the   explicit representation for $G^{(\alpha-\lambda)}(x)$, (\ref{G1}) and \eqref{Hb20},   that
 $$
 \big( K_{t-s}* K_s\big)(x)\leq C K_t(x), \quad |x|\leq 1.
 $$
 Consider now the case  $|x|>1$.   Split
 $$
\big( K_{t-s}* K_s\big)(x) \leq    \Big(\int_{|z|\geq |x|/2} + \int_{|x-z|\geq |x|/2}\Big)K_{t-s}(z)K_s(x-z)dz\normal.
$$
Note that $K_t(x)$ is a monotone function of $|x|$. In addition, it depends on $|x|$ in a piece-wise power-type way, and therefore possesses the
same property formulated in statement 1 of Proposition \ref{A1} for the function $G^{(\lambda)}$. Then
$$
K_{t-s}(z) \leq K_{t-s}(x/2)\leq c K_{t-s}(x), \quad  |z|\geq |x|/2.
$$
For $|x|\geq 1$ we have $K_t(x)= t^{1-\lambda/\alpha} |x|^{-d-\alpha}$, and thus
$$
K_{t-s}(x)=(t-s)^{1-\lambda/\alpha} |x|^{-d-\alpha}
\leq (t-s)^{1-\lambda/\alpha} t^{-1+\lambda/\alpha} K_t(x)\leq  K_t(x),  \quad |x|\geq 1.
$$
Then for $|x|\geq 1$
$$\ba
\int_{|z|\geq |x|/2} K_{t-s}(z)K_s(y-z)\,dz&\leq c K_t(x) \int_{|z|\geq |x|/2}K_s(y-z)\,dz
\\&\leq c K_t(x) \int_{\Re^d}K_s(z')\,dz'\leq C K_t(x),
\ea
$$
where in the last inequality we used \eqref{Hb10} and    \eqref{Psi1}\normal. Similarly,
$$
\int_{|y-z|\geq |x|/2} K_{t-s}(z)K_s(y-z)dz\leq C K_t(x), \quad |x|\geq 1.
$$
Summarizing the estimates proved above, we derive the required sub-convolution property for  $H_t(x,y)$.

\medskip

\emph{Case \textbf{B}. }  Denote for $q\in[0,1]$
$$
K_t^{(q)}(x,y)= K_t\Big(y-qb(x)t-(1-q)  b(y) t-x\Big),
$$
  where $K_t(x)$ is defined in \eqref{Kt}\normal. Observe that now $H_t(x,y)= K_t^{(0)}(x,y)$.

Let us prove the following statement: For a given $T>0$ there exist $c, C$ such that for every $q\in [0,1]$
\begin{equation}\label{key}
c K^{(q)}_t (x,y) \leq K_t(y-tb(y)-x)\leq C K^{(q)}_t (x,y), \quad t\in (0, T].
\end{equation}
We prove only  the first inequality,  the proof of the second  one is completely analogous. Consider two cases: $|x-y|>2Bt$ and $|x-y|\leq 2Bt$, where $B=\sup_x|b(x)|$. In the first case, we have
\begin{equation}\label{xy}
|y-x-b(y)t|\geq {1\over 2}|y-x|, \quad |y-x-qb(x)t-(1-q)b(y)t|\leq {3\over 2}|y-x|.
\end{equation}
Then by the analogue of (\ref{G1}) for $K_t(x)$ we get the first inequality in \eqref{key}.

Consider the case $|x-y|\leq 2Bt$. Then we have
\begin{equation}\label{bxy}
|b(x)-b(y)|\leq c t^\gamma.
\end{equation}
Since  in  case \textbf{B} we  have $1+\gamma>1/\alpha$, by the triangle inequality   we get
\begin{align*}
\Big|\frac{y-z-tb(y)}{t^{1/\alpha}}\Big|^\lambda \vee 1 &\leq C\Big(\Big|\frac{y-z-tb(y) -qt(b(x)-b(y))}{t^{1/\alpha}}\Big|^\lambda \vee 1\Big) + C\\
&\leq C \Big( \Big|\frac{y-z-tb(y) -qt(b(x)-b(y))}{t^{1/\alpha}}\Big|^\lambda \vee 1\Big),\quad t\leq T.
\end{align*}
Note that  for any $C\geq 1, A\geq 0$
$$
(CA)\vee t^{-1/\alpha}\leq C(A\vee t^{-1/\alpha}),
$$
hence  we can finalize the above estimate in the following way:
$$
\frac{K^{(0)}_t(x,y)}{K_t^{(q)}(x,y)} \leq C
{g^{(\alpha-\lambda)}(v)\over g^{(\alpha-\lambda)}(u)},
$$
where
$$
u={y-x-b(y)t\over t^{1/\alpha}}, \quad v={y-x-qb(x)t-(1-q)b(y)t\over t^{1/\alpha}}.
$$
Note that the logarithmic derivative of $g^{(\alpha-\lambda)}(x)$ is bounded; see (\ref{g_a}), (\ref{g_a_der}).
Then
\begin{equation}\label{gl}
{g^{(\alpha-\lambda)}(v)\over g^{(\alpha-\lambda)}(u)}\leq e^{c|u-v|}, \quad u,v\in \Re^d,
\end{equation}
which implies
\begin{equation}\label{k-es1}
\frac{K_t^{(0)}(x,y)}{K_t^{(q)}(x,y)} \leq   C\exp\left[c q|b(x)-b(y)|t^{-1/\alpha+1}\right].
\end{equation}
Since  for  $|x-y|\leq 2Bt$ we have \eqref{bxy},
we can estimate the right-hand side of \eqref{k-es1}  by
$C \exp \left[  ct^{-1/\alpha+1+ \gamma} \right]$,
which is bounded for $t\in [0,T]$  since   $\alpha>(1+\gamma)^{-1}$. This completes the proof of (\ref{key}).

Now we can finalize the proof in  case \textbf{B}.  By \eqref{key} with $q=1$,
\begin{equation}\label{rhs}\ba
\int_\rd H_{t-s}(x,z)& H_s(z,y)\,dz=\int_\rd K_{t-s}^{(0)} (x,z) K_s^{(0)} (z,y)\,dz
\\&\leq C \int_\rd K_{t-s}^{(1)} (x,z) K_s^{(0)} (z,y)\,dz=C \int_\rd K_{t-s}(z-x') K_s (y'-z)\,dz,
\ea
\end{equation}
where
 $$
x'=x+(t-s)b(x),\quad  y'=y-sb(y).
 $$
 The  sub-convolution property of the kernel $K_t(y-x)$ was actually shown in the proof of case \textbf{A}, hence
 \begin{equation}\label{hrs2}
 \int_\rd H_{t-s}(x,z) H_s(z,y)\,dz\leq C K_t(y'-x')=CK^{(1-s/t)}_t(x,y).
 \end{equation}
Applying \eqref{key} with $q=1-s/t$, we complete the proof of the required sub-convolution property for $H_t(x,y)$ in case \textbf{B}. Finally, applying \eqref{key} with $q=1$ we get estimates \eqref{Hint}.

\medskip
\emph{Case } \textbf{C}. The scheme of the proof in this case is similar to that one in the case \textbf{B}. Denote  for  $q\in [0,1]$
$$
\tilde  K_t^{(q)}(x,y)=K_t\Big(\chi_{qt}(\theta_t(y))- \chi_{qt}(x)\Big).
$$
Observe that $\tilde  K_t^{(0)}(x,y)\equiv K_t (\theta_t(y)-x)$    is equal to the kernel $H_t(x,y)$    in the case \textbf{C}.
As in the case \textbf{B}, let us show that
\begin{equation}\label{tilko}
c \tilde{K}^{(q)}(x,y)\leq K_t (\theta_t(y)-x) \leq C \tilde{K}^{(q)}(x,y), \quad t\in (0,T].
\end{equation}
Suppose first that $b\in C_b^1(\Re^d)$. In this case, every
$\chi_t(x)$ is differentiable in $x$, and the respective derivative $D_t(x):=\nabla_x \chi_t(x)$ satisfies the following linear ODE (cf. \cite[Ch.~I, (7.12)--(7.14)]{CL55})
$$
{d\over dt}D_t(x)=B(t,x)D_t(x), \quad B(t,x):=\big(\nabla b\big)(\chi_t(x)).
$$
In addition, $D_0(x)$ is the identity matrix.  Similar relations hold true for the inverse flow, since $\theta_t$ is the solution to \eqref{flow2},  which differs from \eqref{flow1} by the sign "-". Then $\nabla_x \theta_t(x)= D^{-1}_t(x)$, where $D_t^{-1}(x)$ is the inverse matrix of $D_t(x)$, and
$$
{d\over dt}D_t^{-1}(x)=\tilde{B}(t,x)D_t^{-1}(x), \quad \tilde{B}(t,x):=-\big(\nabla b\big)(\theta_t(x)).
$$
 Hence,  we have the following bounds for the matrix norms of $D_t(x)$  and $D_t^{-1}(x)$:
\begin{equation}\label{der}
\|D_t(x)\| \leq C_{b,T}, \quad \left\|\Big(D_t(x)\Big)^{-1}\right\| \leq C_{b,T}, \quad t\in (0,T].
\end{equation}
Note that the constant $C_{b,T}$ depends only on $T$ and on the supremum of the matrix norm of $\nabla b$.
Using these inequalities, we derive
$$
 C_{b,T}^{-1}|\theta_t(y)-x| \leq |\chi_{qt}(\theta_t(y))-\chi_{qt}(x)|\leq  C_{b,T}|\theta_t(y)-x|, \quad q\in [0,1],\, t\in (0,T].
$$
 Then since $\theta_t(y)-x = \big(\chi_{qt}(\theta_t(y))-\chi_{qt}(x)\big)\big|_{q=0}$, we derive  \eqref{tilko} by the property of $K_t$ (cf. the explanation in the case \textbf{A}). Note that
 $$
 \theta_s(y)-\chi_{t-s}(x)=\chi_{t-s}\big(\theta_t(y)\big)-\chi_{t-s}(x)= \chi_{(1-s/t)t}\big(\theta_t(y)\big)-\chi_{(1-s/t)t}(x).
 $$
Then we derive \eqref{rhs}  and \eqref{hrs2}  with $\tilde{K}_{t-s}^{(q)}(\cdot,\cdot)$ instead of $K_t^{(q)}(\cdot,\cdot)$, with $q=0,1$ and $1-s/t$, respectively, and
$$
x'= \chi_{t-s}(x), \quad y'= \theta_s(y).
$$
Then
applying finally \eqref{tilko} with $q=1-s/t$ we derive the sub-convolution property of $H_t(x,y)$.   Applying \eqref{tilko} with $q=1$ we get \eqref{Hint}, which finalized the proof of the proposition in case \textbf{C} if  $b\in C^1(\rd)$.

To handle the Lipschitz case   one can approximate $b$ uniformly by a sequence of functions
$b_n\in C^1_b(\Re^n)$ in such a way that the matrix norms of $\nabla b_n$ remain uniformly bounded.
\end{proof}

\begin{proof}[Proof of Proposition~\ref{H-bound-a}]  a) Without loss of generality assume that $f\in C_\infty(\rd)$ is non-negative.
Then in case \textbf{A} we have
$$
\int_\rd g^{(\alpha)}_t(y-x)f(y)dy=\int_\rd g^{(\alpha)}_t(z)f(z+x)dz\to 0, \quad |x|\to \infty.   $$
In case \textbf{B}  we have by \eqref{key}
\begin{equation}\label{ca2}
\begin{split}
\int_\rd g_t^{(\alpha)}(y-tb(y)-x)f(y)dy
&\leq  C \int_\rd g_t^{(\alpha)}(y-tb(x)-x) f(y) dy\\
&= C \int_\rd g_t^{(\alpha)}(z) f(z+x+tb(x)) dz\to 0, \quad |x|\to \infty,
\end{split}
\end{equation}
   since      $b(\cdot)$ is bounded, and $f\in C_\infty(\rd)$.

Analogously, in case \textbf{C} we have
\begin{align*}
\int_\rd g_t^{(\alpha)}(\theta_t(y)-x)f(y) dy & \leq C \int_\rd g_t^{(\alpha)}(y-\chi_t(x))f(y)dy\\
&= C \int_\rd g_t^{(\alpha)}(z)f(z+\chi_t(x))dz \to 0, \quad |x|\to \infty,
\end{align*}
because $|\chi_t(x)|=\big|x+\int_0^t b(\chi_s(x))ds\big|\to \infty$, $|x|\to\infty$,    since the function $b(\cdot)$ is bounded.

b)
In case \textbf{A} the statement follows from  the fact that $g_t^{(\alpha)}(y-x)$ is the fundamental solution to the Cauchy problem for $\partial_t-L^{(\alpha)}$, in particular,
\begin{equation}\label{fa}
\sup_{x\in \rd} \left|\int_\rd g_t^{(\alpha)}(y-x)f(y)dy-f(x)\right|\to 0 \quad \text{as}\quad t\to 0.
 \end{equation}
  In case \textbf{B} we have
\begin{equation}\label{ca3}
\begin{split}
\Big|\int_\rd g_t^{(\alpha)}&(y-tb(y)-x)f(y)dy- f(x)\Big|
\\
&\leq C\int_\rd \left|g_t^{(\alpha)}(y-tb(y)-x)-g_t^{(\alpha)}(y-tb(x)-x)\right| dy\\
 &\quad +\Big| \int_\rd g_t^{(\alpha)}(y-x-tb(x))f(y)dy- f(x)\Big|=: J_1(t,x)+J_2(t,x).   \end{split}
\end{equation}
Note that by \eqref{g_a_der}
\begin{align*}
\Big|g_t^{(\alpha)}(y-tb(y)-x)&-g_t^{(\alpha)}(y-tb(x)-x)\Big|\\
&\leq  t
\int_0^1 \left|\frac{(b(y)-b(x))}{t^{1/\alpha}} \frac{1}{t^{d/\alpha}} \Big( \nabla g^{(\alpha)}\Big) \left(\frac{y-x-tb(x)-st(b(y)-b(x))}{t^{1/\alpha}}\right)\right| ds\\
&\leq C t \int_0^1
\frac{|y-x|^\gamma}{t^{1/\alpha}} \frac{1}{t^{d/\alpha}} G^{(\alpha+1)} \left(\frac{y-x-tb(x)-st(b(y)-b(x))}{t^{1/\alpha}}\right) ds\\
&\leq C t
\frac{|y-x|^\gamma}{t^{1/\alpha}} \frac{1}{t^{d/\alpha}} G^{(\alpha+1)} \left(\frac{y-x-tb(x)}{t^{1/\alpha}}\right).
\end{align*}
where in the last line we used that the estimate  \eqref{key} also holds true for $\frac{1}{t^{d/\alpha}}G^{(\alpha+1)}\Big(\frac{\cdot}{t^{1/\alpha}}\Big) $ instead of $K_t(\cdot)$.  Using the triangle inequality,  \eqref{G2} and \eqref{g_a}, we derive
$$
\Big|g_t^{(\alpha)}(y-tb(y)-x)-g_t^{(\alpha)}(y-tb(x)-x)\Big|\leq C \big(t^{1+\gamma -1/\alpha}+ t^{1-1/\alpha+\gamma/\alpha}\big) g_t^{(\alpha)}(y-x-tb(x)).
$$
Since in case \textbf{B} we assumed that $\alpha>(1+\gamma)^{-1}$, we have $\alpha >1-\gamma$. Thus,
$$
\sup_x J_1(t,x)\leq C \big(t^{1+\gamma -1/\alpha}+ t^{1-1/\alpha+\gamma/\alpha}\big) \to 0, \quad t\to 0.
$$

For $J_2(t,x) $ we have under the additional assumption that $f\in C^1_\infty(\rd)$
\begin{align*}
J_2(t,x)=\left|\int_\rd g_t^{(\alpha)}(z) \big(f(z+x+tb(x))-f(x)\big)dy\right| \leq C t
\to 0, \quad t\to 0,
\end{align*}
uniformly in $x$. The general case $f\in C_\infty(\rd)$ follows by the approximation argument. This  completes the proof in case \textbf{B}.

In case \textbf{C} the argument is similar. We split
\begin{equation}\label{ca4}
\begin{split}
\Big|\int_\rd g_t^{(\alpha)}&(\theta_t(y)-x)f(y)dy- f(x)\Big|
\\
&\leq C\int_\rd \left|g_t^{(\alpha)}(\theta_t(y)-x)-g_t^{(\alpha)}(y-\chi_t(x))\right| dy\\
 &\quad +\Big| \int_\rd g_t^{(\alpha)}(y-\chi_t(x))f(y)dy- f(x)\Big|=: J_1(t,x)+J_2(t,x).   \end{split}
\end{equation}
Using \eqref{g_a_der}, \eqref{g_a} and \eqref{tilko}, we get
\begin{align*}
\Big|g_t^{(\alpha)}(\theta_t(y)-x)&-g_t^{(\alpha)}(y-\chi_t(x))\Big|\\
&\leq t \int_0^1 \left|\frac{b(\chi_{qt}(\theta_t(y)))-b(\chi_{qt}(x))}{t^{1/\alpha}} \frac{1}{t^{d/\alpha}}  \Big( \nabla g^{(\alpha)}\Big) \left(\frac{\chi_{qt}(\theta_t(y))-\chi_{qt}(x)}{t^{1/\alpha}}\right)\right|\,dq\\
&\leq C t  g_t^{(\alpha)}(y-\chi_t(x)),
\end{align*}
which implies that $\sup_x J_1(t,x)\to 0$ as $t\to 0$.

For $J_2(t,x) $ we have by the same  argument as in case   \textbf{B}
\begin{align*}
J_2(t,x)=\left|\int_\rd g_t^{(\alpha)}(z) \big(f(z+\chi_t(x))-f(x)\big)dy\right|\to 0, \quad t\to 0,
\end{align*}
uniformly in $x$ . This finishes the proof in case \textbf{C}.
\end{proof}

\textbf{Acknowledgement.} We thank  thank   K. Bogdan,  A. Kochubei,    A.Kohatsu-Higa  and  R. Schilling, for inspiring discussions  and helpful remarks. We are very grateful to the anonymous referees whose comments and remarks helped us to improve the paper.
  We  gratefully acknowledge the DFG Grant Schi~419/8-1; the first-named author gratefully acknowledges  the Scholarship of the President of Ukraine for young scientists (2012-2014), and the NCN
grant 2014/14/M/ST1/00600.

   \end{document}